\documentclass[12pt,reqno]{article}
\usepackage{amsmath}
\usepackage{amsthm}
\usepackage{amscd}
\usepackage{amssymb}
\usepackage{graphicx}
\usepackage{fullpage}

\newcommand{\GG}{{\mathcal G}}

\newcommand{\Aut}[1]{\mbox{\rm Aut}(#1)}

\newcommand{\Out}[1]{\mbox{\rm Out}(#1)}

\newcommand{\mc}{\mathbb{C}}
\newcommand{\mz}{\mathbb{Z}}

\newcommand{\N}{\ensuremath{\mathbb{N}}}
\newcommand{\Z}{\ensuremath{\mathbb{Z}}}

\newcommand{\R}{\ensuremath{\mathbb{R}}}

\newcommand{\F}[1]{\ensuremath{\mathbb{F}_{#1}}}

\newcommand{\FN}{\F{n}}

\def\tilde{\widetilde}
\def\hat{\widehat}

\makeatother

\newtheorem{theorem}{Theorem}[section]

\newtheorem{proposition}[theorem]{Proposition}
\newtheorem{corollary}[theorem]{Corollary}
\newtheorem{lemma}[theorem]{Lemma}
\theoremstyle{definition}
\newtheorem{definition}[theorem]{Definition}

\newtheorem{remark}[theorem]{Remark}
\newtheorem{convention}[theorem]{Convention}
\newtheorem{aside}[theorem]{Aside}
\newtheorem{properties}[theorem]{Properties}
\newtheorem{definition/remark}[theorem]{Definition-Remark}

\def\strutdepth{\dp\strutbox}
\def \ss{\strut\vadjust{\kern-\strutdepth \sss}}
\def \sss{\vtop to \strutdepth{
\baselineskip\strutdepth\vss\llap{$\diamondsuit\;\;$}\null}}

\def\strutdepth{\dp\strutbox}
\def \sst{\strut\vadjust{\kern-\strutdepth \ssss}}
\def \ssss{\vtop to \strutdepth{
\baselineskip\strutdepth\vss\llap{$\spadesuit\;\;$}\null}}

\def\strutdepth{\dp\strutbox}
\def \ssh{\strut\vadjust{\kern-\strutdepth \sssh}}
\def \sssh{\vtop to \strutdepth{
\baselineskip\strutdepth\vss\llap{$\heartsuit\;\;$}\null}}

\begin{document}

\title{
The mapping torus
group
of a free group automorphism is
hyperbolic relative to the canonical subgroups of polynomial
growth}

\author{F. Gautero, M. Lustig}

\maketitle

\begin{abstract}
We prove that the mapping torus group $\FN \rtimes_{\alpha} \Z$ of
any automorphism $\alpha$ of a free group $\FN$ of finite rank $n
\geq 2$ is weakly hyperbolic relative to the canonical (up to
conjugation) family $\mathcal H(\alpha)$ of subgroups of $\FN$ which
consists of (and contains representatives of all) conjugacy classes
that grow polynomially under iteration of $\alpha$. Furthermore, we
show that $\FN \rtimes_{\alpha} \Z$ is strongly hyperbolic relative
to the mapping torus of the family $\mathcal H(\alpha)$. As an
application, we use a result of Drutu-Sapir to deduce that $\FN
\rtimes_{\alpha} \Z$ has Rapic Decay.
\end{abstract}

\section{Introduction}

Let $\FN$ be a (non-abelian) free group of finite rank $n \geq 2$,
and let $\alpha$ be any automorphism of $\FN$. It is well known
(see \cite{BH} and \cite{LL5})
 that
elements $w \in \FN$ grow either at least exponentially or at most
polynomially, under iteration of $\alpha$. This terminology is
slightly misleading, as in fact it is the translation length
$||w||_{\mathcal A}$ of $w$ on the Cayley tree of $\FN$ with respect
to some basis $\mathcal A$ that is being considered, which is the
same as the word length in $\mathcal{A}^{\pm 1}$ of any cyclically
reduced $w' \in \FN$ conjugate to $w$.

\smallskip

There is a canonical collection of finitely many conjugacy classes
of finitely generated subgroups $H_{1}, \ldots, H_{r}$ in $\FN$
which consist entirely of elements of polynomial growth, and which
has furthermore the property that every polynomially growing element
$w \in \FN$ is conjugate to an element $w' \in \FN$ that belongs to
some of the $H_{i}$. In other words, the set of all polynomially
growing elements of $\FN$ is identical with the union of all
conjugates of the $H_{i}$. For more details see \S 3 below.

\smallskip

This {\em characteristic
family} $\mathcal H(\alpha) = (H_{1}, \ldots, H_{r})$ is
$\alpha$-invariant up to conjugation, and in the mapping torus group
$$
\FN \rtimes_{\alpha} \Z = \, \, <x_1,\ldots,x_n,t \mid t x_i t^{-1}
= \alpha(x_i) \mbox{ for all } i=1,\ldots,n>$$ one can consider
induced mapping torus subgroups $H_{i}^\alpha = H_{i}
\rtimes_{\alpha^{m_{i}}} \Z$, where $m_{i} \geq 1$ is the smallest
exponent such that $\alpha^{m_{i}}(H_{i})$ is conjugate to $H_{i}$.

There is a canonical family $\mathcal{H}_{\alpha}$ of such mapping
torus subgroups, which is uniquely determined, up to conjugation in
$\FN \rtimes_{\alpha} \Z$, by the characteristic family $\mathcal
H(\alpha)$ (see Definition \ref{inducedmappingtorus}).

\begin{theorem}
\label{MainTheorem}
 Let $\alpha \in \Aut{\F{n}}$, let
 $\mathcal{H}(\alpha) = (H_{1}, \ldots,
H_{r})$ be the characteristic family of subgroups of polynomial
$\alpha$-growth, and let $\mathcal{H}_{\alpha}$ be its mapping
torus. Then:

 \begin{enumerate}
   \item[(1)] $\FN \rtimes_{\alpha} \Z$ is weakly
 hyperbolic relative to ${\mathcal H}(\alpha)$.
   \item[(2)] $\FN \rtimes_{\alpha} \Z$ is strongly hyperbolic relative to
   $\mathcal{H}_{\alpha}$.
 \end{enumerate}
\end{theorem}

Here a group $G$ is called {\em weakly hyperbolic} relative to a family
of subgroups $H_{i}$ if the Cayley graph of $G$, with every
left coset
of any of the $H_{i}$ coned off, is a
$\delta$-hyperbolic space (compare Definition \ref{conedgraph}).
We say that $G$ is {\em strongly
hyperbolic} relative to $(H_{1}, \ldots, H_{r})$ if in addition this
coned off Cayley graph is {\em fine}, compare Definition \ref{fine}.
The concept of relatively hyperbolic groups originates from Gromov's
seminal work
\cite{Gromov}.  It has been fundamentally shaped by Farb \cite{Fa} and
Bowditch \cite{Bo}, and it has since then been placed into the core of
geometric group theory in its most present form, by work of several
authors, see for example \cite {Sz}, \cite{Bu} and \cite{Os}.
The relevant facts about relative hyperbolicity are recalled in \S 2
below.

\bigskip

As a consequence of our main theorem we derive the following
corollary, using earlier results of Jolissaint \cite{Jolissaint} and
Drutu-Sapir \cite{DrutuRD}.

\begin{corollary}
\label{RD} For every $\alpha \in \Aut \FN$ the mapping torus group
$\FN \rtimes_{\alpha} \Z$ satisfies the Rapid Decay property.
\end{corollary}

The proof of this corollary, as well as definitions and background about the Rapid Decay property, are given below in \S \ref{RapidDecay}.

\medskip

Another consequence of our main theorem, pointed out to us by M.
Bridson, is an alternative (and perhaps conceptually simpler) proof
of the following recent result:

\begin{theorem}[Bridson-Groves]
\label{bridsongroves} For every $\alpha \in \Aut \FN$ the mapping
torus group $\FN \rtimes_{\alpha} \Z$ satisfies a quadratic
isoperimetric inequality.
\end{theorem}

The proof of this result is given in a sequence of three
 long papers \cite{BG1} \cite{BG2} \cite{BG3},
where a non-trivial amount of technical machinery is developed.
However, a first step is much easier: The special case of the above
theorem where all of $\FN$ has polynomial $\alpha$-growth (compare
also \cite{Ma}). It is shown by Farb \cite{Fa} that, if a group $G$
is strongly hyperbolic relatively to a finite family of subgroups
which all satisfy a quadratic isoperimetric inequality, then $G$
itself satisfies a quadratic isoperimetric inequality. Thus, the
special case of Bridson-Groves' result, together with our Theorem
\ref{MainTheorem}, gives the full strength of Theorem
\ref{bridsongroves}.

\bigskip

This paper has several ``predecessors'': The absolute case, where the
characteristic family $\mathcal H(\alpha)$ is empty, has been proved
by combined work of Bestvina-Feighn \cite{BF}  (see also \cite{Ga0})
and Brinkmann
\cite{Brink}.
In \cite{Ga2} the case of
geometric automorphisms of $\FN$ (i.e. automorphisms induced by
surface homeomorphisms) has been treated. The methods developed
there and in \cite{Ga0}
have been further extended in \cite{Galast} to give a general
combination theorem for relatively hyperbolic groups
(see also \cite{fumiers}). This
combination theorem is a cornerstone in the proof of our main result
stated above; it is quoted in the form needed here as Theorem
\ref{montheoreme}.

\smallskip

The other main ingredient in the proof of Theorem \ref{MainTheorem}
are $\beta$-train track representatives for free group automorphisms
as developed by the second author (see Appendix), presented here in
\S 4 and \S 5 below.  These train track representatives combine
several advantages of earlier such train track representatives,
althought they are to some extent simpler, except that their
universal covering is not a tree.

The bulk of the work in this paper (\S 6 and \S 7) is devoted to
make up for this technical disadvantage:  We introduce and analyze
{\em normalized paths} in $\beta$-train tracks, and we show that
they can be viewed as proper analogues of geodesic segments in a
tree.

In particular, we prove that in the universal covering of a
$\beta$-train track

\begin{enumerate}
\item[(1)]
any two vertices are connected by a unique normalized path, and

\item[(2)]
normalized paths are quasi-geodesics (with respect to both, the
absolute and the relative metric,
see \S 7).
\end{enumerate}

Normalized paths are useful in other contexts as well. In this paper
they constitute the main tool needed to prove the following
proposition.

The precise definition of a {\em relatively hyperbolic} automorphism
is given below in Definition \ref{hyperbolicauto}.

\begin{proposition}
\label{above}
Every automorphism $\alpha \in \Aut \FN$ is hyperbolic
relative to the characteristic family $\mathcal H(\alpha)$ of subgroups of
polynomial $\alpha$-growth.
\end{proposition}

\medskip
\noindent {\em Acknowledgments.} The first author would like to
thank Martin Bridson for helpful and encouraging remarks. The second
author would like to point out that some of the work presented here
has also been inspired by his collaboration with Gilbert Levitt.
Further thanks go to Universit\'e P. C\'ezanne and Universit\'e de
Provence at Aix-Marseille and to the CIRM at Luminy for having
supported during the residential session ``Groups 007'' in February
2007  a 4 week stay of the first author at Marseille.

\section{Relative hyperbolicity}
\label{relative hyperbolicity}

Let $\Gamma$ be a connected, possibly infinite graph. We assume that
every edge $e$ of $\Gamma$ has been given a length $L(e) > 0$. This
makes $\Gamma$ into a metric space. If $\Gamma$ is locally finite,
or if the edge lengths are chosen from a finite subset of $\R$, then
$\Gamma$ is furthermore a {\em geodesic} space, i.e. any two points
are connected by a path that has as length precisely the distance
between its endpoints.

\begin{definition}
\label{fine} A graph $\Gamma$ is called {\em fine} if for every
integer $n \in \N$ any edge $e$ of $\Gamma$ is contained in only
finitely many circuits of length less or equal to $n$. Here a {\em
circuit} is a closed edge path that passes at most once over any
vertex of $\Gamma$.
\end{definition}

Let $G$ be a finitely generated group and let $S \subset G$ be a
finite generating system.
We denote by $\Gamma_S(G)$ the Cayley graph of $G$
with respect to $S$. We define for every edge $e$ the edge length
to be $L(e) = 1$.

\smallskip

Let $\mathcal H = (H_1,\ldots, H_r)$ be a finite family of subgroups
of $G$, where in the context of this paper the $H_{i}$ are usually
finitely generated.

\begin{definition}
\label{conedgraph}
The {\em $\mathcal H$-coned Cayley graph}, denoted by
$\Gamma^{{\mathcal H}}_S(G)$, is the graph obtained from
$\Gamma_S(G)$ as follows:

\begin{enumerate}
  \item We add an {\em exceptional vertex}
  $v(gH_i)$, for
  each
  coset $g H_i$ of any of the
  $H_{i}$.

  \item We add an edge of length $\frac{1}{2}$ connecting
  any vertex
  $g$ of $\Gamma_S(G)$
  to any of the exceptional vertices
  $v(gH_i)$.
\end{enumerate}

We denote by $| \cdot |_{S, \mathcal H}$ the minimal word length on
$G$, with respect to the (possibly infinite) generating system given
by the finite set $S$ together with the union of all the subgroups
in $\mathcal H$. It follows directly from the definition of the
above lengths that for any two non-exceptional vertices $g, h \in
\Gamma^{{\mathcal H}}_S(G)$ the distance is given by:
$$d(g, h) = \,\, | g^{-1} h |_{S, \mathcal H}$$
\end{definition}

\begin{definition}
\label{relativehyperbolicity} \rm Let $G$ be a group with a finite
generating system $S \subset G$, and let $\mathcal H =
(H_1,\ldots,H_r)$ be a finite family of finitely generated subgroups
$H_i$ of $G$.

\begin{enumerate}
\item[(1)]
The group $G$ is {\em weakly hyperbolic relatively to $\mathcal H$}
if the $\mathcal H$-coned Cayley graph
$\Gamma^{{\mathcal H}}_S(G)$ is $\delta$-hyperbolic, for
some $\delta \geq 0$.

\item[(2)]
The group $G$ is {\em strongly hyperbolic relatively to $\mathcal
H$} if the graph $\Gamma^{{\mathcal H}}_S(G)$ is $\delta$-hyperbolic and fine.
\end{enumerate}

It is easy to see that these definitions are independent of the choice
of
the finite generating system
$S$.
\end{definition}

\begin{definition}
\label{malnormalfamily}
A finite family
$\mathcal{H} = (H_1,\ldots,H_r)$ of subgroups of a group $G$
is called {\em malnormal} if:
\begin{enumerate}
  \item[(a)] for any $i \in \{1,\ldots, r\}$ the subgroup
  $H_i$ is malnormal in $G$ \, (i.e. $g^{-1} H_i g \cap H_i = \{1\}$
  for any $g \in G\smallsetminus H_i$), and
  \item[(b)] for any $i,j \in \{1,\cdots,r\}$ with $i \neq j$, and for any $g
  \in G$, one has
$g^{-1} H_i g \cap H_j = \{1\}$.
\end{enumerate}
\end{definition}

This definition is stable with respect to permutation of the
$H_{i}$, or replacing some $H_{i}$ by a conjugate.

 However, we would
like to alert the reader that, contrary to many concepts used in
geometric group theory, malnormality of a subgroup family
$\mathcal{H} = (H_1,\ldots, H_r)$ of a group $G$ is not stable with
respect to the usual modifications of $\mathcal H$ that do not
change the geometry of $G$ relative to $\mathcal H$ up to
quasi-isometry. Such modifications are, for example, (i) the
replacement of some $H_{i}$ by a subgroup of finite index, or (ii)
the addition of a new subgroup $H_{r+1}$ to the family which is
conjugate to a subgroup of some of the ``old'' $H_{i}$, etc.
Malnormality, as can easily been seen, is sensible with respect to
such changes: For example the infinite cyclic group $\Z$ contains
itself as malnormal subgroup, while the finite index subgroup $2\Z
\subset \Z$ is not malnormal. Similarly, we verify directly that
with respect to the standard generating system $S = \{ 1 \}$ the
coned off Cayley graph $\Gamma_{S}^{2\Z}$ is not fine. This
underlines the well known but often not clearly expressed fact that
the notion of strong relative hyperbolicity (i.e.
``$\delta$-hyperbolic + fine'') is not invariant under
quasi-isometry of the coned off Cayley graphs (compare also
\cite{DRUTU}), contrary to the otherwise less useful notion of weak
relative hyperbolicity.

\medskip

The following lemma holds for any hyperbolic group $G$, compare
\cite{Bo}. In the case used here, where $G = \FN$ is a free group,
the proof is indeed an exercise.

\begin{lemma}
\label{classique}
Let $G$ be a hyperbolic group, and let
$\mathcal{H} = (H_1,\ldots,H_r)$ be a finite family of finitely
generated subgroups.

\smallskip
\noindent
(1)
If the family $\mathcal H$
consists of quasi-convex subgroups, then
 $G$ is weakly hyperbolic relative to $\mathcal H$.

 \smallskip
\noindent
(2)
If the family $\mathcal H$
is quasi-convex and malnormal, then
 $G$ is strongly hyperbolic relative to $\mathcal H$.
\qed
\end{lemma}

For any $\alpha \in \Aut{G}$, for any group $G$, a family of
subgroups $\mathcal H = (H_1,\ldots, H_r)$ is called {\em
$\alpha$-invariant up to conjugation} if there is a permutation
$\sigma$ of $\{1, \ldots, r\}$ as well as elements $h_{1}, \ldots,
h_{r} \in G$ such that $\alpha(H_k) = h_k H_{\sigma(k)} h^{-1}_k$
for each $k \in \{1, \ldots, r\}$.

\smallskip

The following notion has been proposed by Gromov \cite{Gromov} in the
absolute case (i.e. all $H_{i}$ are trivial) and generalized
subsequently in \cite{Galast}.

\begin{definition}
\label{hyperbolicauto}
Let $G$ be a group generated by a finite subset $S$, and let $\mathcal H$ be a
finite family of subgroups of $G$. An automorphism $\alpha$ of $G$ is {\em
hyperbolic relative to $\mathcal H$}, if ${\mathcal H}$ is
$\alpha$-invariant up to conjugation and if there exist constants $\lambda
> 1, M \geq 0$ and $N \geq 1$
such that for any $w \in G$ with $| w |_{S, \mathcal H} \, \, \geq
M$ one has:
$$\lambda | w |_{S, \mathcal H} \, \, \, \leq \, \, \,
\mathrm{max}\{\, | \alpha^{N}(w) |_{S, \mathcal H}\, , \, |
\alpha^{-N}(w) |_{S, \mathcal H}\,\}$$
\end{definition}

The concept of a relatively hyperbolic automorphism is a fairly
``stable'' one, as shown by the following remark:

\begin{remark}
\label{stablerelhyp} Let $G, S, \mathcal{H}$ and $\alpha$ be as in
Definition \ref{hyperbolicauto}. The following  statements can be
derived directly from this definition.

\begin{enumerate}
\item[(a)]
The condition stated in Definition \ref{hyperbolicauto} is independent
of the particular choice of the finite generating system $S$.

\item[(b)]
The automorphism $\alpha$ is hyperbolic relative to $\mathcal H$ if
and only if $\alpha^m$ is hyperbolic relative to $\mathcal H$, for any
integer $m \geq 1$.

\item[(c)]
The automorphism $\alpha$ is hyperbolic relative to $\mathcal H$ if
and only if
$\alpha' = \iota_{v} \circ \alpha$ is hyperbolic relative to $\mathcal
H$, for any inner automorphisms
$\iota_{v}: \FN \to \FN,  w \mapsto v w v^{-1}$.

\end{enumerate}
\end{remark}

Every automorphism $\alpha$ of any group $G$ defines a semi-direct
product
$$G_{\alpha} = G \rtimes_{\alpha} \mz
=  G \, * < t > / << t g t^{-1} = \alpha(g)
\mbox{ for all } g \in G >>$$
which is called the {\em mapping torus group} of $\alpha$.
In our case, where $G = \FN$, one has
$$G_{\alpha} =
\FN \rtimes_{\alpha} \mz = \, \, <x_1,\ldots,x_n,t \mid t x_i t^{-1}
= \alpha(x_i) \mbox{ for all } i=1,\ldots,n>$$ It is well known and
easy to see that this group depends, up to isomorphisms which leave
the subgroup $G \subset G_{\alpha}$ elementwise fixed, only on the
outer automorphism defined by $\alpha$.

Let ${\mathcal H} = (H_1,\ldots,H_r)$ be a finite family of
subgroups of $G$ which is $\alpha$-invariant up to conjugacy. For
each $H_{i}$ in $\mathcal H$ let $m_{i} \geq 1$ be the smallest
integer such that $\alpha^{m_{i}}(H_{i})$ is conjugate in $G$ to
$H_{i}$, and let $h_{i}$ be the conjugator: $\alpha^{m_{i}}(H_{i}) =
h_{i} H_{i} h_{i}^{-1}$. We define the {\em induced mapping torus
subgroup}:
$$
H^\alpha_{i} =\, \,
<H_{i}, h^{-1}_i t^{m_i} > \, \, \subset  \,  G_{\alpha}$$

It is not hard to show that two subgroups $H_{i}$ and $H_{j}$ of $G$
are, up to conjugation in $G$, in the same $\alpha$-orbit if and
only if the two induced mapping torus subgroups $H^\alpha_{i}$ and
$H^\alpha_{j}$ are conjugate in the mapping torus subgroup
$G_{\alpha}$. (Note also that in a topological realization of
$G_{\alpha}$, for example as a fibered 3-manifold, the induced
fibered submanifolds, over an invariant collection of disjoint
subspaces with fundamental groups $H_{i}$, correspond precisely to
the conjugacy classes of the $H^\alpha_{i}$.)

\begin{definition}
\label{inducedmappingtorus} Let ${\mathcal H} = (H_1,\ldots,H_r)$ be
a finite family of subgroups of $G$ which is $\alpha$-invariant up
to conjugacy. A family of induced mapping torus subgroups
$${\mathcal H}_\alpha =
(H^\alpha_{1}, \ldots, H^\alpha_{q})$$ as above is the {\em mapping
torus of $\mathcal H$ with respect to $\alpha$} if it contains for
each conjugacy class in $G_{\alpha}$ of any $H^\alpha_{i}$, for $i =
1, \ldots, r$, precisely one representative.
\end{definition}

\smallskip

The following Combination Theorem has been proved by the
first author. For a
reproof using somewhat different methods compare also \cite{fumiers}.

\begin{theorem}[\cite{Galast}]
\label{montheoreme} Let $G$ be a finitely generated group, let
$\alpha \in \Aut{G}$ be an automorphism, and let $G_{\alpha} = G
\rtimes_{\alpha} \mz$ be the mapping torus group of $\alpha$. Let
${\mathcal H} = (H_1,\ldots, H_r)$ be a finite family of finitely
generated subgroups of $G$, and suppose that $\alpha$ is hyperbolic
relative to $\mathcal H$.

\begin{enumerate}
\item[(a)]
If $G$ is weakly hyperbolic relative to $\mathcal H$, then
$G_\alpha$ is weakly hyperbolic relative to $\mathcal H$.

\item[(b)]
If $G$ is strongly hyperbolic relative to $\mathcal H$, then
$G_{\alpha}$ is strongly hyperbolic relative to
the
mapping
torus
${\mathcal H}_\alpha$ of $\mathcal H$ with respect to $\alpha$.
\end{enumerate}
\end{theorem}

\section{Polynomial growth subgroups}
\label{polynomialgrowth}

Let $\alpha \in \Aut{\F{n}}$ be an automorphism of $\FN$. A subgroup
$H$ of $\F{n}$ is {\em of polynomial $\alpha$-growth} if every
element $w \in H$ is {\em of polynomial $\alpha$-growth}: there are
constants $C > 0,\,  d \geq 0$ such that the inequality
$$|| \alpha^t(w) || \, \, \, \leq \, \, \, C t^d $$
holds for all integers $t \geq 1$, where $|| w ||$ denotes the
cyclic length of $w$ with respect to some basis of $\FN\,$. Of
course, passing over to another basis (or, for the matter, to any
other finite generating system of $\FN$) only affects the constant
$C$ in the above inequality.

\smallskip

We verify easily that, if $H \subset \FN$ is a subgroup of
polynomial $\alpha$-growth, then it
is also of polynomial $\beta^k$-growth, for any $k \in \mz$ and
 any $\beta \in \Aut{\F{n}}$ that represents the same outer
automorphisms as $\alpha$. Also, any conjugate subgroup $H' = g H
g^{-1}$ is also of polynomial growth.

\medskip

A family of polynomially growing subgroups
${\mathcal H} = (H_1,\cdots,H_r)$ is called {\em exhaustive} if
every element $g \in \FN$ of polynomial growth is conjugate to an
element contained in some of the $H_{i}$. The family $\mathcal H$ is
called {\em minimal} if no $H_{i}$ is a subgroup of any
conjugate of some $H_{j}$ with $i \neq j$.

\smallskip

The following proposition
is well known (compare \cite{gjll}).
For completeness we state
it
in full generality, although some
ingredients (for example ``very small'' actions) are not specifically
used here. The paper \cite{lu2} may serve as an introductionary text
for the objects concerned.

\begin{proposition}
\label{invarianttree}
Let  $\alpha \in \Aut \FN$ be an arbitrary automorphism of $\FN$.
Then either the whole group $\FN$ is of polynomial $\alpha$-growth, or
else there is a very small action of $\FN$ on some $\R$-tree $T$ by
isometries, which has the following properties:

\begin{enumerate}
\item[(a)]
The $\FN$-action on $T$ is $\alpha$-invariant with respect to a
stretching factor $\lambda > 1\,$: one has
$$|| \alpha(w) ||_{T} \, \, \, = \, \, \,
\lambda || w ||_{T}$$ for all $w \in \FN$, where $|| w ||_{T}$
denotes the translation length of $w$ on $T$, i.e. the value given
by $|| w ||_{T} \, \, := \inf \{d(wx, x) \mid x \in T \}$.

\item[(b)]
The stabilizer in $\FN$ of any non-degenerate arc in $T$ is trivial:
$$Stab([x, y]) = \{ 1 \}  \qquad \hbox{for all} \qquad x \neq y \in T$$

\item[(c)]
There are only finitely many orbits $\FN \cdot x$ of points $x \in
T$ with non-trivial stabilizer $Stab(x) \subset \FN$. In particular,
the family of such stabilizers $H_{k} = Stab(x_{i})$, obtained by
choosing an arbitrary point $x_{i}$ in each of these finitely many
$\FN$-orbits, is $\alpha$-invariant up to conjugation.

\item[(d)]
For every $x \in T$ the rank of the
point stabilizer $Stab(x)$ is strictly smaller than $n$.
\end{enumerate}
\end{proposition}

We now define a finite iterative procedure, in order to identify all
elements in $\FN$ which have polynomial $\alpha$-growth: One applies
Proposition \ref{invarianttree} again to the non-trivial point
stabilizers $H_{k}$ as exhibited in part (c) of this proposition,
where $\alpha$ is replaced by the restriction to $H_{k}$ of a
suitable power of $\alpha$, composed with an inner automorphism of
$\FN$. By Property (d) of Proposition \ref{invarianttree}, after
finitely many iterations this procedure must stop, and thus one
obtains a partially ordered finite collection of such invariant
$\R$-trees $T_{j}$.

In every tree $T_{j}$ which is minimal in this
collection,
we
choose a point in each of the finitely many orbits with non-trivial
stabilizer, to obtain a finite family $\mathcal H$
of finitely generated
subgroups $H_{i}$ of $\FN$. It follows directly from this definition
that every $H_{i}$ has
polynomial $\alpha$-growth, and that the family $\mathcal H$
is $\alpha$-invariant up to conjugation.

\smallskip

The family $\mathcal H$ is exhaustive, as, in each of the $T_{j}$,
any path of non-zero length grows exponentially, by property (a) of
Proposition \ref{invarianttree}. From property (b) we derive the
minimality of $\mathcal H$: Indeed, we obtain the stronger property,
that any two conjugates of distinct $H_{i}$ can intersect only in
the trivial subgroup $\{1\}$ (see Proposition \ref{malnormal}).

\smallskip

It follows that the family $\mathcal H$ is uniquely determined (by
exhaustiveness and minimality), contrary to the above collection of
invariant trees $T_{j}$, which is non-unique, as the tree $T$ in
Proposition \ref{invarianttree} is in general not uniquely
determined by $\alpha$. The different choices, however, are well
understood: a brief survey of the underlying structural
analysis of
$\alpha$ is given in \S \ref{structure-of-autos} of the Appendix.

\smallskip

We
summarize:

\begin{proposition}
\label{polynomialfamily}
(a)  Every automorphism $\alpha \in \Aut{\F{n}}$ possesses
a finite family
${\mathcal H}(\alpha) = (H_1, \ldots , H_r)$
of finitely generated
subgroups $H_{i}$ that are of polynomial growth, and ${\mathcal H}(\alpha)$
is
exhaustive and minimal.

\smallskip
\noindent
(b)  The family ${\mathcal H}(\alpha)$
is uniquely determined, up to permuting the $H_{i}$ or
replacing any $H_{i}$ by a conjugate.

\smallskip
\noindent (c)  The family ${\mathcal H}(\alpha)$ is
$\alpha$-invariant. \qed
\end{proposition}

The family ${\mathcal H}(\alpha) = (H_1,\cdots,H_r)$
exhibited by Proposition
\ref{polynomialfamily} is called the {\em characteristic family of
polynomial growth} for $\alpha$.
This terminology is slightly exaggerated, as the $H_{i}$ are really
only well determined up to conjugacy in
$\FN$. But on the other hand, the whole
concept of a group $G$ relative to a finite family of subgroups
$H_{i}$ is in
reality a concept of $G$ relative to a conjugacy class of subgroups
$H_{i}$, and it is only for notational simplicity that one prefers to
name the subgroups $H_{i}$ rather than their conjugacy classes.

\begin{proposition}
\label{malnormal} For every automorphism $\alpha \in \Aut{\F{n}}$
the characteristic family of polynomially growing subgroups
$\mathcal H(\alpha)$ is quasi-convex and malnormal.
\end{proposition}

\begin{proof}
The quasi-convexity is a direct consequence of the fact that the
subgroups in ${\mathcal F}(\alpha)$ are finitely generated: Indeed,
every finitely generated subgroup of a free group is quasi-convex,
as is well known and easy to prove.

To prove malnormality of the family $\mathcal H(\alpha)$ we first
observe directly from Definition \ref{malnormalfamily} that if
$\mathcal{H'} = (H'_{1}, \ldots, H'_{s})$ is a malnormal family of
subgroups of some group $G$, and for each $j \in \{1, \ldots, s\}$ one
has within $H'_{j}$
a family of subgroups $\mathcal{H}''_{j} = (H''_{j, 1},
\ldots, H''_{j, r(j)})$ which is malnormal with respect to
$H'_{j}$, then the total family
$$\mathcal{H} = (H''_{j,k})_{(j, k) \in \{1, \ldots, s\} \times
\{1, \ldots, r(j)\} }$$
is a family of subgroups that is malnormal in $G$.

A second observation, also elementary, shows that given any $\R$-tree
$T$ with isometric $G$-action that has
trivial arc stabilizers, every finite system of points $x_{1},
\ldots, x_{r} \in T$ which lie in pairwise distinct $G$-orbits gives
rise to a family of subgroups $(Stab(x_{1}), \ldots, Stab(x_{r}))$
which is malnormal in $G$.

These two observations, together with Proposition
\ref{invarianttree}, give directly the claimed malnormality of the
characteristic family of polynomial $\alpha$-growth.
\end{proof}

\section{$\beta$-train tracks}

A new kind of train track maps $f: \GG^2 \to \GG^2$, called {\it
partial train track maps with Nielsen faces}, has been introduced.
by the second author (see Appendix). Here $\GG^2$ consists of

\begin{enumerate}
\item[(a)]
a disjoint union $X$ (called {\em the relative part}) of finitely
many {\it vertex spaces} $X_v$,

\item[(b)]
a finite collection $\hat \Gamma$
(called {\em the
train track part})
of edges
$e_j$ with endpoints in the $X_v$, and

\item[(c)]
a finite collection of 2-cells $\Delta_k$
with boundary in $\GG^1 := X \cup \hat \Gamma$.
\end{enumerate}

The map $f$ maps $X$ to $X$ and $\GG^1$ to $\GG^1$.  A path
$\gamma_{0}$ in $\GG^1$ is called a {\em relative backtracking path}
if $\gamma_{0}$ is in $\GG^{2}$ homotopic rel. endpoints to a path
entirely contained in $X$. A path $\gamma$ in $\GG^{1}$ is said to
be {\em relatively reduced} if any relative backtracking subpath of
$\gamma$ is contained in $X$.

\begin{convention}
\label{conventiononpaths}
(1)
Note that throughout this paper
we will only consider paths $\gamma$ that are immersed except possibly at the
vertices of $\GG^{2}$.
(Recall that by hypothesis (b) above all
vertices of $\GG^{2}$ belong to $X$.)
In other words,
$\gamma$ is either a classical edge
path, or else an edge path with first
and/or
 last edge that is only
partially traversed.
In the latter case, however, we require that this partially traversed edge belongs
to $\hat \Gamma$.

\smallskip
\noindent
(2)
Furthermore, for subpaths $\chi$ of $\gamma$ that are entirely
contained in $X$, we are only interested in the homotopy class in
$X$ relative endpoints.
In the context considered in this paper, $X$ will always be a graph,
so that we can (and will tacitly) assume throughout the remainder of
the paper that such $\chi$ is a reduced path in the graph $X$.

\smallskip
\noindent
(3)
We denote by $\overline \gamma$ the path
$\gamma$ with inverted orientation.
\end{convention}

In particular, it follows
from convention (2)
that every relatively reduced path $\gamma$
as above is reduced in the
classical sense, when viewed as path in
the graph $\GG^{1}$.  The converse is wrong, because of the 2-cells
$\Delta_{k}$ in $\GG^{2}$.
(Compare also part (b) of Definition-Remark \ref{legalturns}, and the
subsequent discussion.)

\smallskip

A path $\gamma$ in $\GG^1$ is called {\em legal} if for all $t \geq
1$ the path $f^t(\gamma)$ is relatively reduced. The space $\GG^2$
and the map $f$ satisfy furthermore the following properties:

\begin{itemize}
\item
The map $f$ has the {\it partial train track property relative to
$X$}: every edge $e$ of the train track part $\hat \Gamma$ is legal.

\item
Every edge $e$ from the train track part $\hat
\Gamma$ is {\em expanding}: there is a positive iterate of
$f$ which maps $e$ to an edge path that runs over at least two edges
from the train track part.

\item
For every
path (or loop) $\gamma$ in $\GG^1$ there is an integer $t = t(\gamma)
\geq 0$ such that $f^t(\gamma)$ is
homopopic rel. endpoints
(or freely homotopic) in $\GG^2$ to a
legal path (or loop) in $\GG^1$.
\end{itemize}

We say that $f: \GG^2 \to \GG^2$ {\it represents} an automorphism
$\alpha$ of $F_n$ if there is a {\it marking} isomorphism $\theta:
\pi_1 \GG^2 \to F_n$ which conjugates the induced morphism $f_*:
\pi_1 \GG^2 \to \pi_1 \GG^2$ to the outer automorphism $\hat \alpha$
given by $\alpha$.

\medskip

Building on deep work of Bestvina-Handel \cite{BH}, the second
author has shown \cite{Lu1} that every automorphism $\alpha$ of
$F_{n}$ has a partial train track representative with Nielsen faces
$f: \GG^2 \to \GG^2$, and all conjugacy classes represented by loops
in the relative part have polynomial $\alpha$-growth.
However, for the purpose of this paper an additional property is
needed, which in \cite{Lu1}, \cite{Lu2} only occurs for the ``top
stratum'' of $\hat \Gamma$, namely that legal paths lift to
quasi-geodesics in the universal covering $\tilde \GG^{2}$.

This is the reason why one needs to work here with {\em
$\beta$-train track maps}, and with {\em strongly legal} paths,
which have this additional property. This improvement, and some
other technical properties of $\beta$-train tracks needed later are
presented in detail in the next section.

The following result is presented in the Appendix. Note that all
properties of $\beta$-train tracks maps which are used below are
explicitly listed here.

\begin{theorem}
\label{betterttrepresentative} Every automorphism $\alpha$ of $F_n$
is represented by a $\beta$-train track map. This is a partial train
track map with Nielsen faces $f: \GG^2 \to \GG^2$ relative to a
subspace $X \subset \GG^{2}$, which satisfies:

\smallskip
\noindent (a)  Every connected component $X_v$ of $X$ is a graph,
and the marking isomorphism $\theta: \pi_1 \GG^2 \to F_n$ induces a
monomorphism $\pi_1 X_v \to F_n$. Every conjugacy class represented
by a loop in $X$ has polynomial growth.

\smallskip
\noindent (b)  There is a subgraph $\Gamma \subset \GG^1$, which
contains all of the train track part $\hat \Gamma$, and there is a
homotopy equivalence $r: \GG^2 \to \Gamma$ which restricts to the
identity on $\Gamma$, such that the composition-restriction
$f_\Gamma = r \circ f \mid_\Gamma: \Gamma \to \Gamma$ is a classical
relative train track map as defined in \cite{BH}.

\smallskip
\noindent (c)  Every edge $e$ of the train track part of $\GG^{2}$
is strongly legal (see Definition \ref{stronglylegal}) and thus in
particular legal.

\smallskip
\noindent
(d)  Every strongly legal path in $\GG^{1}$ is mapped by $f$ to a
strongly legal path.

\smallskip
\noindent (e) Every edge $e$ from the train track part $\hat \Gamma$
is expanding: there is a positive iterate of $f$ which maps $e$ to
an edge path that runs over at least two edges from $\hat \Gamma$.

\smallskip
\noindent
(f)  The lift of any strongly legal path $\gamma$ to the universal covering
$\tilde \GG^{2}$ is a quasi-geodesic with respect to the simplicial
metric on $\tilde \GG^{2}$ (where every edge in either, the train
track and the relative part, is given length 1), for some fixed
quasi-geodesy constants independent of the choice of $\gamma$.

\smallskip
\noindent (g) Every reduced path in $\Gamma$ lifts also to a
quasi-geodesic in $\tilde \GG^{2}$. Every path that is mapped by the
retraction $r$ to a reduced path in $\Gamma$ lifts also to a
quasi-geodesic in $\tilde \GG^{2}$. In particular, every path which
derives from a strongly legal path by applying $r$ to any collection
of subpaths does lift to a quasi-geodesic in $\tilde \GG^{2}$.

\smallskip
\noindent (h)  For every path $\gamma$ in $\GG^1$ there is an
integer $\hat t = \hat t(\gamma) \geq 0$ such that $f^{\hat
t}(\gamma)$ is homotopic rel. endpoints in $\GG^2$ to a strongly
legal path in $\GG^1$. The integer $\hat t(\gamma)$ depends only on
the number of illegal turns (compare Definition \ref{stronglylegal})
in $\gamma$ and not on $\gamma$ itself.
\end{theorem}

\medskip

For further use of $\beta$-train track maps, in particular with
respect to a structural analysis of automorphisms of $\FN$, we refer
the reader to \S \ref{structure-of-autos} of the Appendix.

\section{Strongly legal paths and INP's in \\ $\beta$-train tracks}

Let $f: \GG^{2} \to \GG^{2}$ be a $\beta$-train track map as
described in the previous section. Recall from the beginning of the
last section that a path $\gamma$ in $\GG^{1}$ is {\em legal} if,
for any $t \geq 1$, the image path $f^t(\gamma)$ is relatively
reduced, i.e. every  relative backtracking subpath of $f^t(\gamma)$
is completely contained in the relative part $X \subset \GG^{1}$.
For the precise definition of a ``path'' recall Convention
\ref{conventiononpaths}.

\begin{definition}
\label{INPpaths} An INP is a reduced path $\eta = \eta' \circ
\eta''$  in $\GG^{1}$ which has the following properties:
\begin{enumerate}
\item[(0)]
The first and the last edge (or non-trivial
edge segment)
of the path $\eta$ belongs to the train track part $\hat \Gamma \subset \GG^{1}$.
\item[(1)]
The
subpaths
$\eta'$ and $\eta''$
(called the {\em branches} of $\eta$)
are
legal.
\item[(2)] The path $f^t(\eta)$ is not
legal, for any $t \geq 0$.
\item[(3)] For some
integer $t_{0} \geq 1$ the path $f^{t_{0}}(\eta)$ is homotopic
relative to its endpoints, in $\GG^{1}$, to the path $\eta$.
\end{enumerate}
\end{definition}

We would like to alert the reader that in the literature one
requires sometimes in property (3) above that $t_0 = 1$, and that
for $t_0 \geq 2$ one speaks of a {\em periodic INP}. We will not
make this notational distinction in this paper.

\smallskip

For every INP $\eta$ there is an associated {\em auxiliary edge} $e$
in the relative part $X \subset \GG^{1}$ which has the same
endpoints as $\eta$. The relative part $X \subset \GG^{1}$ consists
precisely of all auxiliary edges and of all edges $e'$ of $\Gamma
\smallsetminus \hat \Gamma$. In other words: $\GG^{1}$ is the union
of $\Gamma$ with the set of all auxiliary edges.

\smallskip

The canonical retraction $r: \GG^{2} \to \Gamma$ from Theorem
\ref{betterttrepresentative} (b) is given on $\GG^{1}$ as power
$\hat r^n$ of the map $\hat r: \GG^{1} \to \GG^{1}$ which is the
identity on $\Gamma$ and maps every auxiliary edge $e$ to the
associated INP-path $\hat r(e) = \eta$. Recall in this context that
there are only finitely many INP's and thus only finitely many
auxiliary edges, for any given $\beta$-train track map.

\begin{aside}
\label{doubleauxiliary}
Technically speaking, an auxiliary edge $e$ is in truth the union of two
{\em auxilary half-edges}, which meet at an {\em auxiliary vertex}
which is placed in the center of $e$ and
belongs to the relative part. The reason for this particularity
lies in the fact that otherwise 3 (or more) auxiliary edges could form
a non-trivial loop $\gamma$ in $X$ which is contractible in $\GG^{2}$.

To avoid this phenomenon (compare the ``expansion of a Nielsen
face'' in Definition 3.7 of \cite{Lu1}), in this case there is only
one auxiliary vertex which is the common center of the three
auxiliary edges, and only three auxiliary half-edges, arranged in
the shape of a tripod with the auxiliary vertex as center: the union
of any two of the auxiliary half edges defines one of the three
auxiliary edge we started out with. As a consequence, the above loop
$\gamma$ is in fact a contractible loop in the tripod just
described. For more detail and the relation with attractive fixed
points at $\partial \FN$ see \cite{Lu1}, Definition 3.7.
\end{aside}

\begin{definition/remark}
\label{legalturns} (a) A {\em turn} is a path in $\GG^{1}$ of the
type $e \circ \chi \circ e'$, where $e$ and $e'$ are edges (or
non-trivial edge segments) from the train track part $\hat \Gamma$
of $\GG^{1}$, while $\chi$ is an edge path (possibly trivial !)
entirely contained in the relative part $X \subset \GG^{1}$. We
recall (Convention \ref{conventiononpaths}) that one is only
interested in $\chi$ up to homotopy rel. endpoints, within the
subspace $X$, and thus one always assumes that $\chi$ has been
isotoped to be a reduced path in the graph $X$.

\smallskip
\noindent (b) A path $\gamma$ is not legal if and only if for some
$t \geq 1$ the path $f^t(\gamma)$ contains a turn $e \circ \chi
\circ e'$ as above which (i) either is not relatively reduced, i.e.
$\chi$ is a contractible loop and $\bar e = e'$, or else (ii) the
path $\chi$ is (after reduction) an auxiliary edge $e_{0}$ with
associated INP $\hat r(e_{0}) = \eta$, such that $\eta$ starts in
$\bar e$ and ends in $\bar e'$.

\smallskip
\noindent (c) A path $\gamma$ in $\GG^{1}$ is legal if and only if
all of its turns are legal. In particular, every legal path is
relatively reduced (and thus reduced in the graph $\GG^{1}$, see
Convention \ref{conventiononpaths}). The converse implication is
false.

\smallskip
\noindent
(d)
Every INP $\eta = \eta' \circ \eta''$ as in Definition \ref{INPpaths}
has precisely one turn that is not legal, called the {\em tip}
of $\eta$. This is the turn from the last train track edge of $\eta'$
to the first train track edge of $\eta''$.
More specifically, for all $t\geq 1$ the path $f^t(\eta)$ contains
precisely one turn
(= the turn from the last train track edge of $f^t(\eta')$
to the first train track edge of $f^t(\eta'')$)
that is not relatively reduced, as above in alternative (i) of part (b).

\end{definition/remark}

Although not needed in the sequel, we would like to explain the case
(ii) of part (b) above:

For some sufficiently large exponent $t' \geq 1$ there will be a
terminal segment $e_{1}$ of $e$ and an initial segment $e'_{1}$ of
$e'$ with $f^{t'}(e_{1} \circ e_{0} \circ e'_{1}) = \bar \eta' \circ
e_{0} \circ \bar \eta''$. Thus the subpath $f^{t'}(e_{1} \circ e_{0}
\circ e'_{1})$ of $f^{t+t'}(\gamma)$, while not contained in $X$, is
relatively backtracking, since $\bar \eta' \circ e_{0} \circ \bar
\eta''$ is contractible in $\GG^{2}$. This is because $\GG^{2}$
contains for every auxiliary edge $e_{0}$ a 2-cell $\Delta_{e_{0}}$
(called a {\em Nielsen face}) with boundary path $\bar e_{0} \circ
\eta$.  By definition it follows that $\gamma$ is not legal.

\begin{definition/remark}
\label{stronglylegal} (a) A {\em half turn} is a path in $\GG^{1}$
of the type $e \circ \chi$ or $\chi \circ e'$, where $e$ and $e'$
are edges (or non-trivial edge segments) from the train track part
$\hat \Gamma$ of $\GG^{1}$, while $\chi$ is a non-trivial reduced
edge path entirely contained in the relative part $X \subset
\GG^{1}$.

Every finite path $\gamma$
contains only finitely many maximal
(as subpaths of $\gamma$)
half turns,
namely precisely two at each turn, plus a further half turn at the
beginning and another one at the end of $\gamma$.

\smallskip
\noindent (b) A path $\gamma$ in $\GG^{1}$ is called {\em strongly
legal} if it is legal (and thus reduced in $\GG^{1}$), and if in
addition it has the following property: The path $\gamma'$, obtained
from $\gamma$ through replacing every auxiliary edge $e_{i}$ on
$\gamma$ by the associated INP $\eta_{i} = \hat r(e_{i})$, contains
as only illegal turns the tips of the INPs $\eta_{i}$.

\smallskip
\noindent (c) A legal (and hence reduced) path $\gamma$ in $\GG^{1}$
is strongly legal if and only if all maximal half turns in $\gamma$
are strongly legal. A half turn $e \circ \chi$ (or similarly $\chi
\circ e'$) in $\gamma$ is not strongly legal if and only if the
first edge of $\chi$ is an auxiliary edge $e'$ with $\hat r(e') =
\eta$, and for some $t \geq 1$ the first edge of $f^t(\eta)$ is
precisely the first edge of the legal path $f^t(\bar e)$.

\smallskip
\noindent
(d)
A turn is called {\em illegal} if it is not legal, or if any of its
two
maximal
sub-half-turns is not strongly legal.
\end{definition/remark}

\medskip

We will now treat explicitely a technical subtlety which is
relevant for the next section:

If $\eta$ is an INP in $\GG^{1}$, decomposed as above into two legal
(actually they trun out to be always strongly legal !)
{branches} $\eta = \eta' \circ \eta''$, then it can happen that
$\eta'$ (or $\eta''$) contains an auxiliary edge $e_{1}$. Replacing now
$e_{1}$ by its associated INP $\hat r(e_{1}) = \eta_{1}$, the same
phenomenon can occur again: the legal branches of $\eta_{1}$ may well
run over an auxiliary edge. However, this process can repeat only a
finite number of times.

This is the reason why above we distinguish between an INP $\eta =
\hat r(e)$ with associated auxiliary edge $e$ on one hand, and the
path $r(e)$ in $\Gamma \subset \GG^{1}$ obtained through finitely
iteration of $\hat r$ on the other hand.
For any auxiliary edge $e$ we call the reduced
path $r(e)$ in $\Gamma$ a {\em pre-INP}, and we observe that such a
pre-INP may well contain another such pre-INP as subpath (although not
as boundary subpath,
by property (0) of Definition \ref{INPpaths}).

\begin{definition}
\label{preINP}
A pre-INP $r(e)$
in a reduced path $\gamma$ in $\Gamma$
is called {\em isolated}, if any other pre-INP in $\gamma$ that
intersects $r(e)$
in more than a point
is contained as subpath in $r(e)$.
\end{definition}

Clearly,
replacing each such isolated pre-INP $r(e_{i})$ of $\gamma$ by the
associated auxiliary edge $e_{i}$ yields a
path $\gamma'$ in $\GG^{1}$ which does not depend on the order in
which these replacements are performed, and is thus uniquely
determined by $\gamma$.  It also satisfies $r(\gamma') = \gamma$,
which is a reduced path, by hypothesis. Such
a path $\gamma'$ is called a {\em normalized} path in $\GG^{1}$; they
will be investigated more thoroughly in the next section.

\section{Normalized paths in $\beta$-train tracks}

Throughout this section we assume that a $\beta$-train track map $f:
\GG^{2} \to \GG^{2}$ is given as defined in the previous two
sections, and that $f$ represents an automorphism $\alpha$ of $\FN$.
We will use in this section both, the {\em absolute} and the {\em
relative} length of a path $\gamma$ in $\GG^{1}$: The absolute
length $| \gamma |_{abs}$ is given by associating to every edge $e$
of $\GG^{1}$, i.e. of $\hat \Gamma$ and of $X$, the length $L(e) =
1$. The relative length $| \gamma |_{rel}$ is given by associating
to every edge $e$ in the train track part $\hat \Gamma \subset
\GG^{1}$ the length $L(e) = 1$, while every edge $e'$ in the
relative part $X \subset \GG^{1}$ is given length $L(e') = 0$.

\smallskip

We will now start with our study of normalized paths. The reader
should keep in mind that lifts of normalized paths to the universal
cover $\tilde \GG^{2}$ of $\GG^{2}$ are meant (and shown below)
to be strong analogues of geodesic segments in a tree. For example,
one can see directly from the definition that a normalized path is
reduced in $\GG^{1}$ and relatively reduced in $\GG^{2}$, and that a
concatenation of normalized paths, even if not normalized, is
necessarily relatively reduced in $\GG^{2}$ if it is reduced in $\GG^{1}$.

\begin{definition}
\label{normalized} \rm A path $\gamma$ in $\GG^{1}$ is {\em
normalized}, if

\smallskip
\noindent
(i) the path $r(\gamma)$ in $\Gamma$ is reduced, and

\smallskip
\noindent (ii) the path $\gamma$ is obtained from $r(\gamma)$
through replacing every isolated pre-INP $r(e)$ of $r(\gamma)$ by
the associated auxiliary edge $e$.
\end{definition}

\begin{proposition}
\label{uniquenormalized} For every path $\gamma$ in $\GG^{1}$ there
is a unique normalized path $\gamma_{*}$ in $\GG^{1}$ which is (in
$\GG^{2}$) homotopic to $\gamma$ relative to its endpoints.
\end{proposition}

\begin{proof}
To prove existence, it suffices to apply the retraction $r$ to
$\gamma$, followed by a subsequent reduction, to get a reduced path
in $\Gamma \subset \GG^{1}$ that is homotopic rel. endpoints in
$\GG^{2}$ to $\gamma$. One then replaces iteratively every isolated
pre-INP by the associated auxiliary edge to get $\gamma_{*}$.

Since reduced paths in $\Gamma$ are uniquely determined with repect
to homotopy rel. endpoints,
to prove uniqueness of $\gamma_{*}$
we only have to verify that for every normalized
path the above explained procedure reproduces the original path.  This
follows directly from the definition of an ``isolated''
pre-INP at the end of \S 5.
\end{proof}

For an arbitrary path $\gamma$ in the graph $\GG^{1}$ we always
denote by $\gamma_{*}$ the normalized path obtained from $\gamma$ as
given in Proposition \ref{uniquenormalized}.

\begin{proposition}
\label{legalandrelative}
Let $f: \GG^{2} \to \GG^{2}$ be a $\beta$-train track map.

\smallskip
\noindent
(a)  Every stronly legal path $\gamma$ in $\GG^{1}$ is normalised.

\smallskip
\noindent
(b)  If $\gamma$ is a path in $\GG^{1}$ that is entirely contained in the
relative part $X \subset \GG^{2}$, then the normalized path
$\gamma_{*}$ is also entirely contained in $X$.

\smallskip
\noindent
(c)
Normalized paths lift in the universal covering $\tilde
\GG^{2}$ to quasi-geodesics, with respect to the absolute metric on
$\tilde \GG^{2}$.
\end{proposition}

\begin{proof}
Statement (a) follows directly from the above definition of a
normalized path. For statement (b) the same is true, but we also
need the subtlety involved when introducing the auxiliary edges that
has been indicated in Aside \ref{doubleauxiliary}. Part (c) follows
directly from Theorem \ref{betterttrepresentative} (g).
\end{proof}

\begin{lemma}
\label{compositionofnormalized}
There exists a ``composition constant'' $E > 0$ which has the
following property:

\smallskip
\noindent (1) Let $\gamma_{1}$ and $\gamma_{2}$ be two normalized
paths in $\GG^{1}$, and let $\gamma = \gamma_{1} \circ \gamma_{2}$
be the (possibly non-reduced or non-normalized) concatenation.
Then there are decompositions
$\gamma_{1} = \gamma'_{1} \circ \gamma''_{1}$ and
$\gamma_{2} = \gamma''_{2} \circ \gamma'_{2}$
such that
the normalized path $\gamma_{*}$
can be written as concatenation
$$\gamma_{*} = \gamma'_{1} \circ \gamma_{1,2} \circ \gamma'_{2}\, ,$$
where the path
$\gamma_{1, 2}$ has absolute
length
$$ | \gamma_{1, 2} |_{abs} \, \, \leq E \, .$$

\smallskip
\noindent
(2)
If one assumes that the concatenation
$\gamma = \gamma_{1} \circ \gamma_{2}$ is reduced, then one can
furthermore conclude that also the paths
$\gamma''_{1}$ and $\gamma''_{2}$ have absolute
length $\leq E$.

\end{lemma}

\begin{proof}
(1) We first observe that by definition of normalized paths the two
paths $r(\gamma_{1})$ and $r(\gamma_{2})$ in $\Gamma$ are reduced.
Hence there is an initial subpath $r_{1}$ of $r(\gamma_{1})$ as well
as a terminal subpath $r_{2}$ of $r(\gamma_{2})$ such that the
(possibly non-reduced) concatenation $r(\gamma_{1}) \circ
r(\gamma_{2})$ of the reduced paths $r(\gamma_{i})$ can be
simplified to give the reduced path $r_{1} \circ r_{2}$. The claim
now is a direct consequence of the following observation: Any
pre-INP $r(e)$ in the subpaths $r_{i}$ is isolated in $r_{i}$ if and
only if it is isolated in the concatenation $r_{1} \circ r_{2}$,
unless $r(e)$ is contained in a neighborhood of the concatenation
point. But the seize of this neighborhood only depends on the
maximal absolute length of any pre-INP in $\GG^{1}$ and is hence
independent of the particular paths considered.

\smallskip
\noindent (2) In order to prove the stronger claim (2) it suffices
to show that, if the concatenation $\gamma_{1} \circ \gamma_{2}$ is
reduced, then the possible cancellation in $r(\gamma_{1}) \circ
r(\gamma_{2})$ is bounded.

By way of contradiction, assume that the reduced paths
$\overline{r(\gamma_{1})}$ (= the path $r(\gamma_{1})$ with
orientation reversed) and $r(\gamma_{2})$ have a long common initial
segment $\gamma_{0}$. By the argument given above in part (a), the
occurrences of isolated pre-INP's in $\gamma_{0}$, other than in a
terminal subsegment of $\gamma_{0}$ of a priory bounded length, do
not depend on whether we consider the segment $\gamma_{0}$ as part
of $\overline{r(\gamma_{1})}$ or of $r(\gamma_{2})$. But then the
normalized paths $\overline\gamma_{1}$ and $\gamma_{2}$ will also
have a long common initial segment, which contradicts the assumption
that $\gamma_{1} \circ \gamma_{2}$ is reduced.
\end{proof}

The following is crucially used in the next section:

\begin{corollary}
\label{oldsevenfour} For any constant $D > 0$ there exists a bound
$K > 0$ which has the following property: Let $\gamma = \gamma_{1}
\circ \gamma_{2} \circ \gamma_{3}$ be a concatenated path in
$\GG^{1}$, and assume that their lengths satisfy:

\smallskip
(i)  \, $| \gamma_{1} |_{abs} \, \, \leq D$

\smallskip
(ii) $| \gamma_{2} |_{rel} \, \, = 0$

\smallskip
(iii) $| \gamma_{3} |_{abs} \, \, \leq D$

\smallskip
\noindent
Then the normalized path $\gamma_{*}$ has relative length
$$| \gamma_{*} |_{rel} \, \, \leq K \, .$$
\end{corollary}

\begin{proof}
We consider the normalized paths $\gamma_{i*}$ and observe that, by
Proposition \ref{legalandrelative} (c),  the absolute length of
$\gamma_{1*}$ and $\gamma_{3*}$ is bounded above by a constant only
dependent on $D$. Furthermore, the relative length is always smaller
or equal to the absolute one. Hence the sum of the relative lengths
of the $\gamma_{i*}$ depends only on $D$, and Lemma
\ref{compositionofnormalized} (1) implies directly that the same is
true for the relative length of the normalized path $\gamma_{*}$.
\end{proof}

\begin{lemma}
\label{normalizationproblem1}
There exists a constant $J \geq 1$ such that for any
path $\gamma$ in $\GG^{1}$ the following holds, where $ILT(\cdot)$
denotes the number of illegal turns in a path:

\begin{enumerate}
\item[(a)]
If $\gamma$ is non-reduced, then the path $\gamma'$ obtained from
$\gamma$ through reduction in $\GG^{1}$ satisfies
$$ILT(\gamma') \leq ILT(\gamma)$$

\item[(b)]
If $\gamma$ is reduced, and $\gamma_{*}$ is obtained from $\gamma$
through normalization, then one has:
$$ILT(\gamma_{*}) \leq J \cdot ILT(\gamma)$$
\end{enumerate}
\end{lemma}

\begin{proof}
(a)  This is a direct consequence of the fact that at every turn $e
\circ \chi \circ e'$ of $\gamma$, where $\chi$ is a reduced path in
$X$, either $\gamma$ is reduced (in the graph $\GG^{1}$), or  $\chi$
is trivial and $e' = \bar e$, in which case the turn is illegal.

\smallskip
\noindent
(b)
We first use Proposition \ref{legalandrelative} (b) to observe that
the maximal strongly legal subpaths of $\gamma$ are normalized. We
then use iteratively Lemma \ref{compositionofnormalized} (1) to obtain
$k = ILT(\gamma)$ subpaths $\gamma_{i}$ of $\gamma_{*}$, each
of absolute length
bounded
above
by the constant $E$ from Lemma \ref{compositionofnormalized}
(1), such that every complementary subpath of the union of the $\gamma_{i}$
in $\gamma_{*}$ is strongly legal. But the number of illegal turns in
any $\gamma_{i}$
cannot exceed the absolute length of $\gamma_{i}$,
which gives directly the claim.
\end{proof}

\begin{proposition}
\label{inversegrowth} Let $f: \GG^{2} \to \GG^{2}$ be a
$\beta$-train track map. Then there is a integer $K \geq 1$ such
that for any normalized path $\gamma$ in $\GG^{1}$, the number
$ILT(\cdot)$ of illegal turns satisfies:
$$ILT(\gamma) \, \, \, \geq \, \, \,
2
\,\, \,  ILT(f^{K}(\gamma)_{*})$$
\end{proposition}

\begin{proof}
We first consider any path $\gamma''$
in $\GG^{1}$ with
at most $2J + 1$ illegal turns,
for $J \geq 1$ as given in Lemma \ref{normalizationproblem1}.
By property (h) of Theorem
\ref{betterttrepresentative} there is a constant $K$ such that
$f^{K}(\gamma'')_{*}$
is strongly legal, for all such paths $\gamma''$.

We now subdivide $\gamma$ into $k + 1 \leq \frac{ILT(\gamma)}{2J}$
subpaths such that each subpath has $\leq 2J + 1$ illegal turns. We
consider the normalized $f^{K}$-image of each subpath, which is
strongly legal, and their concatenation $f^{K}(\gamma)$, which
satisfies $ILT(f^{K}(\gamma)) \leq k$, but is a priori not reduced,
and after reduction a priori not normalized. We then apply Lemma
\ref{normalizationproblem1} to obtain $ILT(f^{K}(\gamma)_{*}) \leq J
\cdot k$ and hence $ILT(f^{K}(\gamma)_{*}) \leq  \frac{1}{2}
ILT(\gamma)$.
\end{proof}

\begin{lemma}
\label{boundedcancellation}
There is a ``cancellation bound'' $C = C(f) > 0$ such that for any
concatenated
normalized path $\gamma = \gamma_{1} \circ \gamma_{2}$ the normalized
image path decomposes as $f(\gamma)_{*}  =
\gamma'_{1}  \circ \gamma_{1, 2} \circ \gamma'_{2}$, with
$f(\gamma_{1})_{*} = \gamma'_{1}  \circ \gamma''_{1}$ and
$f(\gamma_{2})_{*} = \gamma''_{2}  \circ \gamma'_{2}$, and all
three, $\gamma_{1, 2},
\gamma''_{1}$ and $\gamma''_{2}$ have length $\leq C$.
\end{lemma}

\begin{proof}
The analogous statement, with every normalized path replaced by its
(reduced !)
image in $\Gamma$ under the retraction $r$, follows directly from the
fact that $f$ represents an automorphism and hence induces a
quasi-isometry on the universal covering of $\GG^{2}$, with respect to
the absolute metric.

To deduce now the desired statement for normalized paths it suffices
to apply the same arguments as in the proof of Lemma
\ref{compositionofnormalized} (2).
\end{proof}

Let $\gamma$ be a path in $\GG^{1}$, and let $C > 0$ be any
constant.  We say that a strongly legal subpath $\gamma'$ of
$\gamma$ has {\em strongly legal $C$-neighborhood} in $\gamma$ if
$\gamma'$ occurs as subpath of a larger strongly legal subpath
$\gamma''$ of $\gamma$ which is of the form $\gamma'' = \gamma_{1}
\circ \gamma' \circ \gamma_{2}$, where each of the $\gamma_{i}$
either has relative length $| \gamma_{i} |_{rel} \, \, = C$, or else
$\gamma_{i}$ is a boundary subpath (possibly of length 0) of
$\gamma$.

In other
words, there is no illegal turn in $\gamma$ that has relative
distance $< C$ within $\gamma$ from the subpath $\gamma'$.

\smallskip

Let  $b(f) \geq 1$ denote the {\em expansion exponent} of $f$,
defined to be the smallest positive exponent such that for any
edge $e \in \hat \Gamma$ the image
$f^{b(f)}(e)$
is an edge path of relative length $\geq 2$.
The existence of $b(f)$ is a direct consequence of statement (e) of
Theorem \ref{betterttrepresentative}.

\begin{proposition}
\label{forwardgrowth} Let $f: \GG^{2} \to \GG^{2}$ be a
$\beta$-train track map, let $b = b(f)$ be the expansion exponent of
$f$, and let $C = C(f^b)$ be the cancellation bound for $f^b$ as
given in Lemma \ref{boundedcancellation}. Then for any normalized
edge path $\gamma$ in $\GG^{1}$ the following holds:

\smallskip
\noindent Every strongly legal subpath $\gamma_{0}$ with strongly
legal $C$-neighborhood in $\gamma$ is mapped by $f^b$ to a strongly
legal path $\gamma'_{0}$ which is contained as subpath with strongly
legal $C$-neighborhood in $f^b(\gamma)_{*}$. Furthermore, their
relative lengths satisfy:
$$| \gamma'_{0} |_{rel} \, \, \geq \, \, 2 | \gamma_{0}
|_{rel}$$
\end{proposition}

\begin{proof}
By definition of the exponent $b$ every strongly legal path
$\gamma_{0}$ is mapped to a path $f^b(\gamma_{0})$ of relative
length $| f^b(\gamma_{0}) |_{rel} \, \, \, \geq \,  2 | \gamma_{0}
|_{rel} \, $.

Now, every strongly legal path is normalized (by Proposition
\ref{legalandrelative} (a)), and the image of a strongly legal path
is again strongly legal (by Theorem \ref{betterttrepresentative}
(d)). Since $\gamma_{0}$ has strongly legal $C$-neighborhood, and
$b$ is the expansion constant of $f$, the path $f^{b}(\gamma_{0}) =
f^{b}(\gamma_{0})_{*}$ has strongly legal $2C$-neighborhood in the
(possibly unreduced and after reduction not normalized) path
$f^b(\gamma)$. But then Lemma \ref{boundedcancellation} proves
directly that in the normalized path $f^b( \gamma)_{*}$ the path
$f^{b}(\gamma_{0})$ has still strongly legal $C$-neighborhood.
\end{proof}

\begin{corollary}
\label{expanding} For every $\lambda > 1$ there exist an integer $N
\geq 1$ such that, if $\gamma$ is a normalized path in $\GG^{1}$
then

\smallskip
\noindent
(a) either the normalized path $f^{N}(\gamma)_{*}$ has
relative
length
$$| f^{N}(\gamma)_{*} |_{rel} \, \, \, \geq
\, \, \, \lambda | \gamma |_{rel}$$

\smallskip
\noindent (b) or else any normalized path $\gamma'$ in $\GG^{1}$
with $f^{N}(\gamma')_{*} = \gamma$ satisfies
$$| \gamma' |_{rel} \, \, \,\geq \, \, \,
\lambda | \gamma |_{rel} \, .$$
\end{corollary}

\begin{proof}
Let $k \geq 0$ be the number of illegal turns in $\gamma$ and set $C
= C(f^b(f))$ as in Proposition \ref{forwardgrowth}. There are
finitely many (at most $k + 1$ ones) maximal strongly legal subpaths
$\gamma_{i}$ with strongly legal $C$-neighorhood in $\gamma$. If $|
\gamma |_{rel} \, \, \, \geq 3C k$, then the total relative length
of the $\gamma_{i}$ exceeds $\frac{1}{3} | \gamma |_{rel}$. Applying
Proposition \ref{forwardgrowth} iteratively to each of them gives
directly the claim (a).

\smallskip

If $| \gamma |_{rel} \, \, \, < 3C k$ we apply iteratively
Proposition \ref{inversegrowth}. Since the relative length of any of
the strongly legal subpath of $\gamma'$ between two adjacent illegal
turns is bounded below by 1 (= the length of any edge in the train
track part), we derive directly the existence of a constant $N \geq
1$ that has the property claimed in statement (b).
\end{proof}

\begin{remark}
Corollary \ref{oldsevenfour} above is used crucially in the next
section. The proof given in this section relies on the particular
properties of normalized paths as introduced in this paper. In this
remark we would like to propose an alternative proof, which dates
back to the original plan for this paper. It is conceptually simpler
in that it doesn't directly appeal to train track technology and to
the very intimate knowledge of normalized paths which we have used
earlier in this section. It uses though some of the later material
of this section, such as Corollary \ref{expanding}. However, the
latter is anyway used crucially in \S 8 below.

We do not present this alternative proof in full detail, but we
believe that the interested reader can recover the latter from the
sketch given here. We first collect the following observations; only
the last one is non-trivial, and none of them uses Corollary
\ref{oldsevenfour}:

\begin{enumerate}
  \item The relative length of a path is smaller or equal to the absolute
  length of the same path.
  \item The map $f$ induces a quasi-isometry for both, the relative
  and absolute metric.
  \item The relative part is
  quasi-convex in the given
  $2$-complex $\tilde \GG^{2}$.
  \item The absolute length of any geodesic in the tree $\widetilde{\Gamma}$
  which connect two points
  in a same connected component of the relative part grows at
  most polynomially, under iteration into the future and also into the past.
  \item Any subpath of a normalized path is normalized, up to adding
  subpaths with absolute length uniformly bounded above at its
  extremities.
  \item The relative length of each normalized path is expanded by a
  factor $\lambda > 1$ after $N$ iterations, either into the future or
  into the past. (This is the content of Corollary \ref{expanding}.)
\end{enumerate}

\noindent We conclude:

\medskip
{\em Let $c$ be a normalized path. If $c^\prime \subset c$ has its
endpoints at absolute distance less than $C$ from two points in a
same connected component of the relative part, then the relative
length of $c^\prime$ is bounded above by some constant depending
only on $C$}. (This is essentially the content of Corollary
\ref{oldsevenfour}.)

\medskip

The proof of this conclusion is not really hard and only requires to
manipulate properly some of the inequalities that are given by the
above observations: One compares the exponential expansion of the
relative length to the polynomial expansion of the absolute length,
which eventually leads to a contradiction, unless the above
conclusion holds.

\end{remark}

\section{Normalized paths are relative quasi-geodesics}

In this section we consider the universal covering $\tilde \GG^{2}$
of $\GG^{2}$ with respect to both, the {\em absolute metric} $d_{abs}$
and the {\em relative metric} $d_{rel}$, which are
defined by lifting the absolute and the
relative edge lengths respectively from $\GG^{2}$ to $\tilde \GG^{2}$.
We also ``lift'' the terminology: for example, the {\em relative part}
of $\tilde \GG^{2}$ is the lift of the relative part $X \subset \GG^{2}$.

\smallskip

We first note that every connected component of the relative
part of $\tilde \GG^{2}$ is quasi-convexly embedded, with respect to
the absolute metric, since the fundamental group of any connected
component of $X$ is a finitely generated subgroup of the free group
$\pi_{1} \GG^{2} = \FN$.

\smallskip

Next, we recall that $\tilde \GG^{2}$, with respect to the
absolute metric, is quasi-isometric to a metric tree: Such a
quasi-isometry is given for example by any lift of the retraction
$r: \GG^{2} \to \Gamma$ to $\tilde r: \tilde \GG^{2} \to \tilde
\Gamma$, which is again a retraction, and
$\tilde \Gamma$ is a metric simplicial tree with free
$\FN$-action.

\smallskip

Finally, let us recall that a path $\gamma$ is a $(\lambda,
\mu)$-quasi-geodesic, for given constants $\lambda > 0$, $\mu \geq
0$, if and only if for every subpath $\gamma^\prime$ of $\gamma$,
with endpoints $x^\prime$ and $y^\prime$, one has:
$$| \gamma' | \, \, \leq \, \, \lambda \,
d(x^\prime,y^\prime) + \mu$$

\begin{proposition}
\label{hausdorff} For all constants $\lambda, \mu > 0$ there are
constants $\lambda', \mu', C > 0$, such that the following holds in
$\tilde \GG^{2}$:

 For every absolute $(\lambda, \mu)$-quasi-geodesic
$\gamma$ there exists a relative $(\lambda', \mu')$-quasi-geodesic
$\hat \gamma$ which is of absolute Hausdorff distance $\leq C$ from
$\gamma$.
\end{proposition}

\begin{proof}
We consider a relative geodesic $\gamma'$ with same endpoints as
$\gamma$, as well as their images $\tilde r(\gamma)$ and $\tilde
r(\gamma')$.  The path $\tilde r(\gamma)$ is contained in an
absolute neighborhood of the geodesic segment $[x, y]$ in the tree
$\tilde \Gamma$, where $x$ and $y$ are the endpoints of $\tilde
r(\gamma)$.

Since $\tilde \Gamma$ is a tree, the path $\tilde
r(\gamma')$ must run over all of $[x, y]$,
so that we can consider a minimal collection of
subpaths $\gamma'_{i}$ of $\gamma'$ such that the union of all
$\tilde r(\gamma'_{i})$ contains the segment $[x, y]$.
(Here ``minimal'' means that no collection of proper subpaths of
the $\gamma'_{i}$ has the
same property).
We note that
the number of such subpaths is bounded above by the absolute length of $[x,
y]$.

We now enlarge these subpaths by a bounded amount, to ensure that
they are edge paths: This ensures that the preimage $\gamma'_{i}$ of
any such $\tilde r(\gamma'_{i})$
is either

(i)
completely contained in the relative part, or else

(ii)
it is of relative
length $\geq 1$.

\noindent Now, the adjacent endpoints of any two subsequent
$\gamma'_{i}$ can be connected by paths $\gamma'_{j}$ of bounded
absolute length in $\tilde \GG^{2}$, and, if the two endpoints
belong to the same connected component of the relative part, then by
the absolute quasi-convexity of the latter we can assume that
$\gamma'_{j}$ as well belongs to this component. In particular, we
observe that the number of paths $\gamma'_{j}$ that are not
contained in the relative part is bounded above by the relative
length of $\gamma'$.

Hence the path $\hat \gamma$, defined as
alternate concatenation of the $\gamma'_{i}$ and
$\gamma'_{j}$, has relative length given as sum of the relative length
of the pairwise dijoint subpaths $\gamma'_{i}$ of the relative
geodesic $\gamma'$, plus the relative length of the $\gamma'_{j}$,
which is uniformly bounded. Since the number of $\gamma'_{j}$ is also
bounded by the relative
length of $\gamma'$, it follows that there are constants as in the
proposition which bound the relative length of $\hat \gamma$.

Since the very same arguments extend to all subpaths of $\hat
\gamma$, it follows  directly that $\hat \gamma$ is a relative
quasi-geodesic as claimed.
\end{proof}

Below we need the following lemma; its proof follows directly from
the definition of a quasi-geodesic and the inequality
$d_{rel}(\cdot, \cdot) \leq d_{abs}(\cdot, \cdot)$.

\begin{lemma}
\label{boundedrelative} For any constants $\lambda, \mu> 0$, every
relative $(\lambda, \mu)$-quasi-geodesic $\gamma$ in $\tilde
\GG^{2}$, which does not traverse any edge from the relative part,
is also an absolute $(\lambda, \mu)$-quasi-geodesic. \qed
\end{lemma}

\begin{proposition}
\label{Pour Martin 2} There exist constants $\lambda, \mu > 0$ such
that in $\tilde \GG^{2}$ any lift $\gamma$ of a normalized path
$\gamma_{0}$ in $\GG^{1}$ is a relative
$(\lambda,\mu)$-quasi-geodesic.
\end{proposition}

\begin{proof}
We note that it suffices to prove:

\begin{enumerate}
\item[(*)]
\label{finito} There exist constants $C_1, C_2, C_{3} \geq 0$ as
well as $\lambda' \geq 1$, $\mu' \geq 0$, such that for any subpath
$\gamma^\prime$ of $\gamma$, with endpoints $x', y'$ (of $\gamma'$),
there exist a relative $(\lambda^\prime,\mu^\prime)$-quasi-geodesic
$\hat \gamma^{\prime}$ with endpoints $\hat x^{\prime}, \hat
y^{\prime}$, such that $d_{rel}(x', \hat x') \leq C_{1}$ and
$d_{rel}(y', \hat y') \leq C_{1}$, and
$$| \gamma^{\prime} |_{rel} \, \, \,  \leq  \, \, C_2 \, \,
| \hat \gamma^{\prime} |_{rel} \, + \, \,  C_{3} \, . $$
\end{enumerate}

By Proposition \ref{legalandrelative} (c), the lift $\gamma$ of the
normalized path $\gamma_{0}$ is an absolute quasi-geodesic, for
quasi-geodesy constants independent of the choice of $\gamma$. Now,
Proposition \ref{hausdorff} gives
 a relative quasi-geodesic $\hat \gamma$ in an absolute Hausdorff
neighborhood of $\gamma$, where the seize of this neighborhood as
well as the quasi-geodesy constants are again independent of $\gamma$.
As a consequence, for any subpath $\gamma'$ of $\gamma$ we find a
corresponding subpath $\hat \gamma'$ of $\hat\gamma$ which satisfies the
endpoint conditions in (*) for some constant $C_{1} > 0$ independent of
our choices.

Without loss of generality we can assume that the path $\hat \gamma$
is contained in the 1-skeleton of $\tilde \GG^{2}$, and that
furthermore $\hat \gamma'$ is an edge path, i.e. starts and ends at
a vertex of $\tilde \GG^{2}$.

\smallskip

We now consider the set $\mathcal{\hat L}$ of maximal subpaths $\hat
\gamma_{i}$ of $\hat\gamma'$ which are contained in the relative
part. The collection of closed subpaths $\hat \gamma_{j}$ of $\hat
\gamma'$ complementary to those in $\mathcal{\hat L}$ is denoted by
$\mathcal{\hat L}^c$.  We observe that, by Lemma
\ref{boundedrelative}, every such $\hat \gamma_{j}$ is an absolute
quasi-geodesic, with quasi-geodesy constants depending only on
$C_{1}$ and not on our choice of $\hat \gamma'$. Furthermore, every
such $\hat \gamma_{j}$ has absolute length $\geq 1$ (= the relative
length of any edge outside the relative part), and we have:
$$| \hat \gamma_{j} |_{abs} \, \, \, = \, \, \,  | \hat
\gamma_{j} |_{rel}$$

The path $\gamma'$ inherits a natural ``decomposition'' $\mathcal{L}
\, \sqcup \, \mathcal{L}^{c}$ from the decomposition of $\hat
\gamma'$ into $\mathcal{\hat L} \, \sqcup \, \mathcal{\hat L}^{c}$:
In order to define the set $\mathcal{L}$, we associate to each
element $\hat \gamma_{i}$ of $\mathcal{\hat L}$ the maximal subpath
$\gamma_{i}$ of $\gamma'$ with the endpoints that are $C_{1}$-close
to the endpoints of $\hat \gamma_{i}$. We now apply Corollary
\ref{oldsevenfour}, to obtain that the relative length of each such
path $\gamma_{i}$ in ${\mathcal L}$ is smaller than some constant
$K > 0$ which is dependent on the seize of $C_{1}$ but independent
of all our choices.

We now define the collection $\mathcal{L}^{c}$ of subpaths of
$\gamma'$ simply as those subpaths $\gamma_{j}$
which connect the endpoints of the
corresponding  subsequent subpaths $\gamma_{i}$ from $\mathcal{
L}$ as defined above.  Of course, the $\gamma_{j}$ may have length 0,
or if two $\gamma_{i}$ overlap, they may run in the opposite direction
than $\gamma'$. But all this does not matter, as the concatenation
of all subsequent paths from $\mathcal{L}$ and $\mathcal{
L}^{c}$ clearly runs through all of $\gamma'$, and hence has bigger or
equal  relative length than $\gamma'$.

Now, by definition, for every path $\gamma_{j}$ in $\mathcal{L}^{c}$
there is a corresponding path $\hat \gamma_{j}$ in $\mathcal{\hat
L}^{c}$ that has endpoints $C_{1}$-close to the endpoints of
$\gamma_{j}$.  Since both, $\gamma_{j}$ and $\hat \gamma_{j}$ are
absolute quasi-geodesics, since the relative length is always
bounded above by the absolute length, i.e. $|  \gamma_{j} |_{rel} \,
\, \leq \,\, | \gamma_{j} |_{abs}\,$, and since we derived above $|
\hat \gamma_{j} |_{rel} \, \, =\, \, | \hat \gamma_{j} |_{abs}\, $,
there are constants $D_{1}, D_{2} > 0$ such that
$$| \gamma_{j} |_{rel} \, \, \, \leq \, \, \, D_{1} \, \,
\cdot | \hat \gamma_{j} |_{rel} \, + \, \, D_{2}$$ But the number of
alternating subpaths from $\mathcal L$ and $\mathcal{L}^c$ is equal
to that of $\mathcal{\hat L}$ and $\mathcal{\hat L}^c$ and thus
bounded above by the relative length of $\hat \gamma'$. Since the
relative length of each $\gamma_{i}$ in $\mathcal L$ is bounded by
the constant $K$, we obtain directly the existence of constants
$C_{1}, C_{2}$ and $C_{3}$ as claimed above in (*).
\end{proof}

\section{Proof of the Main theorem}

We first prove Proposition \ref{above} as stated in the Introduction.
The notion of a
relative hyperbolic automorphism is recalled in Definition \ref{hyperbolicauto}:

\bigskip
\noindent
{\em Proof of Proposition \ref{above}.}

We consider the universal covering
$\tilde \GG^{2}$ of the $\beta$-train track
$\GG^{2}$ from the $\beta$-train track representative $f: \GG^{2}
\to \GG^{2}$ of $\alpha$. We lift the relative length on edges to $\tilde
\GG^{2}$ to make $\tilde \GG^{2}$ into a pseudo-metric space,
and we pass over to the associated metric space $\hat \GG^{2}$ by
contracting every edge of length $0$.
This
amounts
precisely to contracting every connected component $\tilde X_{i}$ of
the full preimage $\tilde X$
of the relative part $X \subset \GG^{2}$ to a single
point $\hat X_{i}$.

\smallskip

We now lift the train track map $f$ to a map $\tilde f: \tilde
\GG^{2} \to \tilde \GG^{2}$ which represents $\alpha$ in the following
sense: For any $w \in \FN$ and any point $P \in \tilde \GG^{2}$ one
has:
$$\alpha(w) \tilde f P = \tilde f w (P)$$
Since $f$ maps $X$ to itself, the map $\tilde f$ induces canonically
a map $\hat f: \hat \GG^{2} \to \hat \GG^{2}$ that satisfies
similarly, for any $w \in \FN$ and any point $P \in \hat \GG^{2} \,$:
$$\alpha(w) \hat f P = \hat f w (P)$$

\smallskip

For our purposes below we also want, in addition to this ``twisted
commutativity property'', that $\hat f$ fixes a vertex of $\hat
\GG^{2}$ outside of the union $\hat X$ of all $\hat X_{i}$. To
ensure this we apply property (e) of Theorem
\ref{betterttrepresentative} and raise $f$ to a sufficiently high
power $f^k$ in order to find a fixed point in the interior of an
edge $e$ of $\hat \Gamma$ (i.e. outside of $X$): We then subdivide
edges finitely many times in order to make this fixed point into a
$f^k$-fixed vertex of $\GG^{2}$. We then lift $f^k$ to the map $\hat
f^k$ constructed above and compose it with the deck transformation
action of a suitable element $v \in \FN$ so that some lift of this
$f^k$-fixed vertex is fixed by $v \hat f^k$. It follows that $v \hat
f^k$ ``twistedly commutes'' with $\iota_{v} \, \alpha^k$ in the
above meaning, where $\iota_{v}$ denotes the inner automorphisms
$\iota_{v}: \FN \to \FN,  w \mapsto v w v^{-1}$.

\smallskip

By virtue of Remark \ref{stablerelhyp} we can continue to work with
$v \hat f^k$ and $\iota_{v} \, \alpha^k$ rather than with $\hat f$
and $\alpha$ as above, without loss of generality in our proof.
However, for simplicity of notation we stick for the rest of the
proof to $\hat f$ and $\alpha$, but we assume that $\hat f$ has a
fixed vertex $Q = \hat f(Q) \in \hat \GG^{2} \smallsetminus \hat X$.

\medskip

We now consider any generating system $S$ of $\FN$, and the associated
coned Cayley graph
$\Gamma^{\mathcal H(\alpha)}_{S}(\FN)$ as given in Definition
\ref{conedgraph}.
We define an $\FN$-equivariant map
$$
\psi:  \Gamma^{\mathcal H(\alpha)}_{S}(\FN) \to \hat \GG^{2}$$ by
sending the base point $V(1)$ to the above $\hat f$-fixed vertex $Q
\in \hat \GG^{2} \smallsetminus \hat X$. Every cone vertex of
$\Gamma^{\mathcal H(\alpha)}_{S}(\FN)$ is mapped to the
corresponding contracted connected component $\hat X_{i}$ of $\hat
X$. The correspondence here is given through the subgroup of $\FN$
which stabilizes a cone vertex of $\Gamma^{\mathcal
H(\alpha)}_{S}(\FN)$, since the same subgroup stabilizes also the
``corresponding'' contracted connected component $\hat X_{i}$ of
$\hat X$. Every edge $e$ of $\Gamma^{\mathcal H(\alpha)}_{S}(\FN)$
is sent to an edge path $\psi(e)$ in $\hat \GG^{2}$ of length
$L(\psi(e)) > 0 \,$: By construction no two distinct vertices of
$\Gamma^{\mathcal H(\alpha)}_{S}(\FN)$ are mapped by $\psi$ to the
same vertex in $\hat \GG^{2}$.

\smallskip

It follows that those edges of
$\Gamma^{\mathcal H(\alpha)}_{S}(\FN)$
that are adjacent to the same cone vertex are
mapped by $\psi$ to edge paths that all have the same length.
It is easy to see directly that the map $\psi$ is a quasi-isometry
(alternatively one can use Proposition 6.1 of \cite{Ga2}). Since
we are only interested in estimating the distance of vertices (which
are mapped by $\psi$ again to vertices), and any distinct two
vertices in either space have distance $\geq \frac{1}{2}$, we can
suppress
the additive
constant in the quasi-isometry inequalities to obtain
%
a constant $C >
0$ such that for all vertices $P, R \in \tilde
\Gamma^{\mathcal H(\alpha)}_{S}(\FN)$
one has:
$$\frac{1}{C} \, \, d(P, R) \leq d(\psi(P), \psi(R)) \leq C d(P, R)$$

Similarly, the canonical inequalities obtained from Proposition
\ref{Pour Martin 2}, which
describe that every normalized path in $\GG^{1}$ lifts to a
quasi-geodesic in $\hat \GG^{2}$,
will only be applied to edge
paths which are either of relative length 0 or are bounded away from 0
by 1 (= the length of the shortest edge in $\hat \GG^{1}$). Hence we obtain directly,
for a suitable constant $A > 0$ and any two vertices $P, R \in
\hat \GG^{2}$ that are connected by a normalized edge path
$\gamma(P, R)$, the inequalities:
$$d(P, R) \, \, \leq\, \, \,  | \gamma(P, R) |_{rel} \, \, \,
\leq \, \,  A \, \, d(P, R)$$

Thus we can calculate, for any $w \in \FN$
and for $\lambda > 0$ as given in Corollary \ref{expanding},
for which we first assume that
alternative (a)
holds:

$$
\begin{array}[t]{rcl}
| w |_{S, \mathcal{H}} \, \, \,  & = & \, \,  d(V(1), V(w)) \\
&\leq & C^{} \, \, \, d(\psi(V(1)), \psi(V(w)))  \\
& \leq & C^{}  | \gamma(\psi(V(1)), \psi(V(w))) |_{rel} \\
&\leq & \frac{C}{\lambda}^{}  | {\tilde f^{N}}(\gamma(\psi(V(1)),
\psi(V(w))))_{*} |_{rel} \\
&\leq &  {\frac{C}{\lambda}}^{{}^{}}  \, \, A \, \, d(\tilde f^{N}(\psi(V(1))),
\tilde f^{N}(\psi(V(w)))) \\
&\leq & \frac{A \, C}{\lambda}^{}   \,  \, \, d(\tilde f^{N}(Q),
\tilde f^{N}(w Q) \\
&\leq & \frac{A \, C}{\lambda}^{}  \, \, \, d(\tilde f^{N}(Q),
\alpha^{N}(w) \tilde f^{N}(Q) \\
&\leq & \frac{A \, C}{\lambda}^{}  \, \, \, d(Q,
\alpha^{N}(w)Q) \\
&\leq & \frac{A \, C}{\lambda}^{}  \, \, \, d(\psi(V(1)),
\psi(V(\alpha^{N}(w)))) \\
& \leq & \frac{A \, C}{\lambda}^{}  \, \, C \, \, d(V(1),
V(\alpha^{N}(w))) \\
&= & \frac{A \, C^{2}}{\lambda}^{}  \, | \alpha^{N}(w) |_{S,
\mathcal{H}}
\end{array}$$
 %
Since the constants $A$ and $C$ are independent of $N$,
a sufficiently large choice of $\lambda$ in Corollary
\ref{expanding} gives the desired conclusion (compare Definition
\ref{hyperbolicauto}).

The calculation for case (b) in Corollary \ref{expanding} is
completely analogous and not carried through here.
The only additional argument to be mentioned here is to ensure the
existence of a path $\gamma'$ as in Corollary \ref{expanding} (b).
But this follows directly from the fact that the $\beta$-train track map
$f: \GG^{2} \to \GG^{2}$
represents an automorphisms of $\FN$, so that we can assume that $f$ (and
thus $\tilde f$) is surjective:  Otherwise one could replace
$\GG^{2}$ by a proper $f$-invariant subcomplex, and the corresponding
restriction of $f$ would be again a $\beta$-train track map which has
otherwise the same
properties as $f$.

\qed
\bigskip

We can now give the proof of the main theorem of this paper as stated in the
Introduction:

 %



\bigskip
\noindent
{\em Proof of Theorem \ref{MainTheorem}.}
From Proposition
\ref{malnormal} we know that $\mathcal{H}(\alpha)$ is quasi-convex
and malnormal. Thus Lemma \ref{classique} implies that $\FN$ is
strongly hyperbolic relative to $\mathcal{H}(\alpha)$.
Furthermore,
from Proposition \ref{above} we know that $\alpha$ is hyperbolic
relative to $\mathcal{H}(\alpha)$.
Hence Theorem
\ref{montheoreme} implies directly the claim.
 %
 %
\qed


\section{The Rapid Decay property}
\label{RapidDecay}

The {\em Rapid Decay property} (or {\em property (RD)}) was originally
established by U. Haagerup for finitely generated free groups
\cite{Haagerup}.
Indeed,
it has
also
been called
``Haagerup inequality'' (compare \cite{Talbi}). The first to formalize and
study systematically property (RD) was Jolissaint in \cite{Jolissaint}. 
Subsequent to his
pioneering work, property (RD) was shown to hold for
various
classes of groups, in particular for hyperbolic groups \cite{delaharpeRD}, certain
classes of groups acting on CAT(0)-complexes \cite{Chatterji},
relatively hyperbolic groups \cite{DrutuRD}, etc.

\smallskip

The main importance of property (RD) comes from its applications
to
the Novikov conjecture and the Baum-Connes conjecture,
see \cite{Connes} and \cite{Lafforgue}. In particular, property (RD) is
useful in constructing explicit isomorphisms in this branch of K-theory.

\medskip

We review the basic definitions and results that we need to prove
Corollary \ref{RD}. Property (RD) may be stated in many equivalent
ways. We borrow from \cite{Jolissaint} the definition which seems to
be the simplest one for
somebody
with expertise in geometric group theory.

\smallskip




Consider any metric $d$ on a group $G$ which is equivariant with respect to the
action of $G$ on itself by left-multiplication.
The function $L : G \to \R, g \mapsto
d(1, g)$ is called a {\em length function on $G$}.

\smallskip

Such a length function, together with the choice of an exponent $s > 0$, is used to define a norm  $|| \cdot ||_{2,s,L}$ on the group algebra $\mc [G]$, which is given for any $\phi = \sum \phi(g) g \in \mc[G]$ by:
$$||\phi||_{2,s,L} = \sqrt{\sum_{g \in G}
\phi(g) \overline{\phi}(g) (1 + L(g))^{2s}}$$
We now consider the Hilbert space $l^2(G)$ and interpret any $ \phi \in\mc [G]$
as linear operator on $l^2(G)$, where the image of any $\psi \in l^2(G)$ is given by the convolution
$\phi*\psi$, defined as usual by $(\phi*\psi)(g) = \underset{h \in G}{\sum} \phi(h)
\psi(h^{-1}g)$.
We can now consider also for any $\phi \in \mc[G]$ the operator norm
$$||\phi|| = \sup_{\psi \in l^2(G)}
\frac{||\phi*\psi||_2}{||\psi||_2}\, ,$$
where
$||\psi||_2 = \sqrt{\underset{g \in G}{\sum} \psi(g)
\overline{\psi}(g)}$
denotes the classical $l_2$-norm of $\psi \in l^2(G)$.

\begin{definition} [\cite{Jolissaint}]
A group $G$ has {\em property (RD) with respect to a length function $L$}
if
there exist
constants $c > 0$ and $s > 0$ such
that, for any $\phi \in \mc [G]$, one has:
$$||\phi|| \leq c \, ||\phi||_{2,s,L}$$
%


\end{definition}

It is proved in \cite{Jolissaint} (Lemma 1.1.4 and Remark 1.1.7)
that, if a group $G$ has property (RD) with respect to the word-length
function given by a finite generating set, then it has property
(RD) with respect to any length function.

\begin{definition}
A finitely generated group $G$ has {\em property (RD)} if it has property
(RD) with respect to the word-length function defined by any
finite generating set of $G$.
\end{definition}

It is shown in \cite{Haagerup} that finitely generated free groups have property (RD).
Hence in the context of this section we don't need to
 refer to any other length
function.

\smallskip

For any homomorphism $\beta$ of a group $G$ with finite generating set $S$ 
one defines
$$a(\beta) = \max_{s \in S} |\beta(s)|_S \, ,$$
where $|g|_S$ denotes the word length of $g \in G$ with respect to $S$.
Following \cite{Jolissaint}, for a second group $\Gamma$ with finite generating system $\Sigma$, a homomorphism $\theta \colon
\Gamma \rightarrow \Aut{G}$ is said to have {\em polynomial amplitude}, if
there
exist positive numbers $c$ and $r$ such that, for any $\gamma \in
\Gamma$ one has:
$$a(\theta(\gamma)) \leq c (1 + |\gamma|_\Sigma)^r$$
This notion is easily seen to be independent of the particular choice of the generating systems $S$ and $\Sigma$.
%
%
The following result has been shown by Jolissaint:

\begin{proposition}
[\cite{Jolissaint}]
\label{resultat jolissaint}
Let $G$ and $\Gamma$ be two finitely generted groups,
and let $\theta \colon
\Gamma \rightarrow \Aut{G}$ be a homomorphism with polynomial amplitude.
%
%
%
If $\Gamma$ and $G$ have property (RD),
then
so does the semi-direct product
$G\rtimes_\theta \Gamma$
 defined by $\theta$.
\end{proposition}


We will also need the following theorem
due to
Drutu-Sapir:

\begin{theorem}
[\cite{DrutuRD}]
\label{resultat Drutu}
Let $G$ be a group which is strongly
hyperbolic relative to a finite family
$\mathcal H$ of finitely generated
subgroups $\mathcal H_j$. If all the subgroups in $\mathcal H_j$ have
property (RD), then so does $G$.
\end{theorem}


\begin{proof}[Proof of Corollary \ref{RD}]
We first consider the special case where $\alpha \in \Aut{\F{n}}$ is an automorphism of
polynomial growth: in this case one deduces directly that the map $\theta \colon \mz \rightarrow
\Aut{\F{n}}$, defined by $\theta(t) = \alpha^t$, has polynomial
amplitude.
Thus it follows from Proposition \ref{resultat jolissaint} that $\F{n} \rtimes_\alpha \mz$
has property (RD).

We now consider an arbitrary automorphism $\alpha \in \Aut{\F{n}}$.
From part (2) of Theorem \ref{MainTheorem}
we know that $\F{n}
\rtimes_\alpha \mz$ is strongly hyperbolic, relative to the canonical
family $\mathcal H_\alpha$ of
mapping torus subgroups
$\mathcal H_j$
over
subgroups $H_j \subset \FN$ where
the restriction of
$\alpha$ has polynomial growth
(see
\S \ref{polynomialgrowth}). By the above argument, each one of
$\mathcal H_j$
has property (RD). Thus we can conclude from Theorem \ref{resultat Drutu} that
$\F{n} \rtimes_\alpha \mz$
as well must have property (RD).
\end{proof}


\bibliographystyle{alpha}



\newpage
\section
{Appendix to
``The mapping torus
group
of a free group automorphism is
hyperbolic relative to the canonical subgroups of polynomial
growth''}

\centerline{\large  by  Martin Lustig}

\bigskip
\medskip

 %
 %
 %
 %
 %

\maketitle

Bestvina-Handel have proved in \cite{BH} that every automorphism of $\FN$ can be represented by a relative train track map $f: \Gamma \to \Gamma$,
recalled below in subsection \ref{rel-tt-to-beta-tt}.
The goal of this appendix is to
explain
how one can derive from such a relative train track map a $\beta$-train track map as defined in section 4.
We give here in a detailed and careful manner all ingredients needed in this construction,
and we sketch the proofs.
A fully expanded version of this appendix will be given in the
forthcoming paper \cite{Lu-better-tts}.

\subsection{Partial train track maps and Nielsen faces}
\label{partial-tt-maps}

A {\em graph-of-spaces $\GG$ relative to $X$} consists of a finite collection $X$ of pathwise connected {\em vertex spaces} $X_v$, and a finite collection $\hat \Gamma$ of edges with endpoints in $X$.  We call $X$ the {\em relative part} of $\GG$, and the edges in $\hat \Gamma$ are referred to as {\em edges of $\GG$} (by which we exclude possible edges in $X$ !).
A subpath $\gamma_0$ of a path $\gamma$ in $\GG$ is called
{\em backtracking}
if the endpoints of $\gamma_0$ coincide, and if the resulting loop
is contractible in $\GG$.
A path $\gamma$ in $\GG$ is called {\em reduced} (rel. $X$) if every backtracking subpath $\gamma_0$ of $\gamma$ is contained in $X$.


\begin{definition}
\label{partial-tr-maps}
Let $\GG$ be a graph-of-spaces with vertex space collection $X \subset \GG$, and
let $f: \GG \to \GG$ be a continuous map
with $f(X) \subset X$.
\begin{enumerate}
\item[(1)]
A path $\gamma$ in $\GG$ is called
{\em legal} (with respect to $f$) if for every $t \geq 1$
the path $f^t(\gamma)$ is reduced.
%
\item[(2)]
The map $f: \GG \to \GG$ is a {\em partial train track map relative to $X$} if every edge of $\GG$ is legal. In this case the collection $\hat \Gamma$ of edges in $\GG$ is
called the {\em train track part} of $\GG$.
\item[(3)]
The map $f$ is called {\em expanding} if for every edge $e \in \hat \Gamma$ some iterate $f^t(e)$
runs over 2 or more edges from $\hat \Gamma$.
\end{enumerate}
\end{definition}


A path $\eta$ in $\GG$ is called an {\em indivisible Nielsen path (INP)}
if $f(\eta)$ is homotopic rel. endpoints to $\eta$, and if $\eta$ is a concatenation $\eta = \gamma \circ \gamma'$ of two legal paths $\gamma$ and $\gamma'$ which are not contained in $X$. Note that
$\eta$ can not be legal,
 if $f$ is expanding. Note also that the endpoints of $\eta$ may be situated in the interior of an edge of $\hat \Gamma$.
A path $\eta$ is called {\em periodic indivisible Nielsen path (periodic INP)} if $\eta$ is an INP for some positive iterate $f^t$ of $f$.
We do not distinguish between periodic INP's that are homotopic via a homotopy
(not necessarily fixing endpoints)
that takes entirely place in the relative part of $\GG$.


\begin{definition}
\label{pseudo-legal}
Let $f: \GG \to \GG$ be a partial train track map relative to
the collection $X$
of vertex spaces of a graph-of-spaces
$\GG$.
\begin{enumerate}
\item[(a)]
A concatenation $\gamma \circ \gamma'$ of two path $\gamma$ and $\gamma'$ in $\GG$ is called
{\em legal concatenation} if
there exists a terminal subpath $\gamma_0$ of $\gamma$ and an initial subpath $\gamma'_0$ of $\gamma'$,
both not
entirely contained in $X$, such that the concatenation $\gamma_0 \circ \gamma'_0$ is legal.
\item[(b)]
A path $\gamma$ in $\GG$ is called {\em pseudo-legal} if $\gamma$ is a legal concatenation of legal paths and periodic INP's.
\end{enumerate}
\end{definition}


It is known that every INP of $f: \GG \to \GG$ defines a branch point orbit in an $\R$-tree with isometric $\pi_1 \GG$-action, which can be obtained from the partial train track via a row-eigen-vector of the geometric transition matrix of $f$ (compare \S 4 of \cite{lu2} and the references given there). If $\pi_1\GG$ is a free group $\FN$ of finite rank $n$, then the number of such branchpoints and their ``multiplicity'' is bounded in terms of $n$, see
\cite{gl}.  As a consequence, one obtains:

\begin{proposition}
\label{finitely-many-INP's}
For any expanding partial train track map $f: \GG \to \GG$, with finitely generated free group $\pi_1 \GG$, there are only finitely many periodic INP's in $\GG$.
\end{proposition}




The following proposition has been shown in \S 3 of \cite{Lu1}.

\begin{proposition}
\label{eventuallypseudolegal}
(a) Let $\GG$ be a graph-of-spaces with vertex space collection $X$, and let $f: \GG \to \GG$ be an expanding partial train track map relative to $X$.
Then for every path $\gamma$ in $\GG$ there is an exponent $t(\gamma) \geq 1$ such that $f^{t(\gamma)}(\gamma)$ is homotopic rel. endpoints to a pseudo-legal path.

\smallskip
\noindent
(b)  There is an upper bound to the exponent $t(\gamma)$, which depends only on the number $q$ of factors in any decomposition of $\gamma = \gamma_1 \circ \ldots \circ \gamma_q$ as concatenation of legal paths $\gamma_i$, and not on the particular choice of $\gamma$ itself.

\end{proposition}


\begin{definition}
\label{Nielsen-faces}
Let $f: \GG \to \GG$ be an expanding partial train track map relative to
the vertex space collection $X$ of a graph-of-spaces $\GG$.
\begin{enumerate}
\item[(a)]
Let $\eta$ be an INP of $f$.
Then gluing an {\em auxiliary edge} $e$ along $\partial e = \partial \eta$ to $\GG$, and simultaneously a {\em Nielsen face}, i.e. a 2-cell $\Delta^2$, along the boundary path $\partial \Delta^2 = \eta^{-1} \circ e$,
is called {\em expanding a Nielsen face at the INP $\eta$}. One extends the map $f$ by the identity on $e$ and by mapping $\Delta^2$ to $\Delta^2 \cup f(\eta)$, to obtain a {\em partial train track map with Nielsen faces $f: \GG^2 \to \GG^2$ relative to $\hat X$}, for the resulting space $\GG^2 = \GG \cup e \cup \Delta^2$ and $\hat X = X \cup e$.
\item[(b)]
Similarly one defines the expansion of an $f$-orbit of Nielsen faces at the $f$-orbit of a periodic INP.
It is also possible
to expand Nielsen faces at several periodic INP's simultaneously.
\item[(c)]
Every Nielsen face $\Delta^2$, expanded together with an auxiliary edge $e$ at some (periodic) INP $\eta$, defines a homotopy which deforms a path that runs over $e$ to a path which runs over $\eta$ instead. Thus the collection of Nielsen faces that have been expanded at (periodic) INP's of $\GG$ defines a strong deformation retraction $\hat r: \GG^2 \to \GG$, which maps every auxiliary edge $e_i$ to the corresponding (periodic) INP $\eta_i$ and leaves every point of $\GG$ fixed.
\item[(d)]
A path $\gamma$ in $\GG^1 = \hat \Gamma \cup \hat X \subset \GG^2$ is called {\em strongly reduced} if $\hat r(\gamma)$ is reduced in $\GG$ (relative to $X$). It follows that in this case $\gamma$ is also reduced, as path in $\GG^1$ relative to $\hat X$, but also as path in $\GG^2$ relative to $\hat X$.  (The subtle difference here is caused by the above definition of a ``backtracking subpath": a subpath of $\gamma$ may well be backtracking in $\GG^2$ but not in $\GG^1$ !)
\item[(e)]
A path $\gamma$ in $\GG^1 \subset \GG^2$ is called {\em strongly legal} if for any $t \geq 1$ the path $f^t(\gamma)$ is strongly reduced. Note that every strongly legal path is legal.
\item[(f)]
The train track map $f: \GG^2 \to \GG^2$ is said to be {\em strong} if every edge of the train track part $\hat \Gamma \subset \GG^2$ is strongly legal.
\end{enumerate}
\end{definition}

\begin{remark}
\label{pseudo-becomes-legal}
Let $f: \GG \to \GG$ be a partial train track map relative to a collection $X \subset \GG$ of vertex spaces, and let $\hat f: \GG^2 \to \GG^2$ be obtained from $f$ and $\GG$ by expanding Nielsen faces at finitely many (periodic) INP's.
%
\begin{enumerate}
\item[(a)]
Then the restriction $f_1: \GG^1 \to \GG^1$ of $\hat f$ to the union $\GG^1$ of $\GG$ with all added auxiliary edges is a partial train track map (without Nielsen faces !) relative to $\hat X$, where the fundamental group $\pi_1 \GG^1$ has been increased with respect to $\pi_1 \GG$ through adding the auxiliary edges to the relative part $X$ to get $\hat X$.

\item[(b)]
If at every periodic INP in $\GG$ a Nielsen face has been expanded to obtain $\GG^2$, then every pseudo-legal path in $\GG$ is homotopic in $\GG^2$ to a strongly legal path.
\end{enumerate}
\end{remark}

For notational purposes we now extend the above introduced notation slightly. Note however that all 2-cells attached to partial train tracks in our context will be Nielsen faces, but in a iterated fashion which
for notational convenience we prefer to suppress.

\begin{definition}
\label{strong-tts}
Let $\GG^1 = \hat \Gamma \cup \hat X$ be a graph-of-spaces relative to $\hat X$, and let $\GG^2$ be obtained from $\GG^1$ by attaching finitely many 2-cells along their boundary to $\GG^1$.  We call $\GG^2$ a {\em graph-of-spaces with 2-cells}.

For some subspace $X \subset \hat X$, which contains all endpoints of edges of $\hat \Gamma$, let $\GG = \hat \Gamma \cup X$ be a graph-of-spaces (without 2-cells !) relative to $X$, and assume that $r: \GG^2 \to \GG$ is a strong deformation retraction.

A map $f: \GG^2 \to \GG^2$ is a {\em strong train track map with 2-cells, relative to $\hat X$}, if
\begin{enumerate}
\item[(a)]
$f(\GG^1) \subset \GG^1$ and $f(\hat X) \subset \hat X$,
\item[(b)]
the induced map $f_1: \GG^1 \to \GG^1$ is an expanding train track map relative to $\hat X$, and
\item[(c)]
every edge $e \in \hat \Gamma$ is {\em strongly legal} with respect to $r\,$:  For any $t \geq 1$ the path $r f_1^t(e)$ is reduced (rel. $X$) in $\GG$.
\end{enumerate}
\end{definition}

\subsection{Building up strong train tracks}

\bigskip

The following lemma is an important tool in the proof of our main result presented in subsection \ref{rel-tt-to-beta-tt}.
A proof appeared already in
\cite{Lu2} (at the end of \S 6),
modulo a minor switch in the terminology employed.
%

\begin{lemma}
\label{one-edge}
Let $f: \GG^2 \to \GG^2$ be a strong partial train track map of a graph-of-spaces $\GG^2$ with 2-cells, relative to a subspace $X \subset \GG^2$.
\begin{enumerate}
\item[(a)]
Let $\GG^2_1$ and $f_1: \GG^2_1 \to \GG^2_1$ be obtained from $\GG^2$ by attaching a further edge $e$ at its endpoints to points $P_0, P_1 \in X$, and by extending $f$ via $f_1(e) = \gamma_0 \circ e \circ \gamma_1$, where the $\gamma_i$ are strongly legal paths in $\GG$. Then one can homotope in $\GG^2$ the attaching points $P_0$ and $P_1$ to points $P'_0$ and $P'_1$, and replace $\gamma_0$ and $\gamma_1$ correspondingly by homotopic paths $\gamma'_0$ and $\gamma'_1$, so that $e$ becomes strongly legal.

In particular, the resulting map $f'_1: \GG'^2_1 \to \GG'^2_1$ is a strong partial train track map with 2-cells, relative to $X$, provided that at least one of the paths $\gamma'_i$ is not entirely contained in $X$
(in order to ensure that $f'_1$ is expanding).
\item[(b)]
The analogous statement is true if $\GG^2_1$ is constructed from $\GG^2$ by attaching finitely many edges $e^j$ to $X$, with $f_1(e^j) = \gamma_0^j \circ e^{\pi(j)} \circ \gamma_1^j$ for some permutation $\pi$.
\item[(c)]
The above given homotopies lead canonically to a homotopy equivalence $h: \GG^2_1 \to \GG'^2_1$, which restricts on $\GG^2$ to a selfmap that is homotopic to the identity and satisfies $h(X) \subset X$, such that $f'_1 h$ and $h f_1$ are homotopic.
\end{enumerate}
\end{lemma}

\begin{definition}
\label{initial-segments}
A partial train track map $f: \GG \to \GG$ relative $X$ satisfies the {\em initial-segments condition} if for every edge $e$ of the train track part of $\GG$ some initial and some terminal segment of $e$ are mapped by $f$ onto an edge of the train track part of $\GG$.
\end{definition}

The following lemma can be derived in a direct manner from the above definitions.

\begin{lemma}
\label{exponential-stratum}
Let $f: \GG \to \GG$ be
a partial train track map of a graph-of-spaces $\GG$ relative to the vertex space collection $X$ of $\GG$, which satisfies the initial-segments condition, and with the property that all edges of $\GG$ are attached to some subspace $X' \subset X$.

Assume that $X$ is itself a graph-of-spaces $X = \GG'^2$ with 2-cells, relative to $X'$, and assume that the restriction $f'$ of $f$ to $X = \GG'^2$ is
a strong partial train track map with 2-cells, relative to $X'$.

Assume furthermore that for every edge $e$ of $\GG$ any subpath $\gamma$ of $f(e)$ that is entirely contained in $X = \GG'$ is strongly legal.

Then $\GG$ is a graph-of-spaces with 2-cells, relative to $X'$, and $f$ is a strong partial train track map with 2-cells, relative to $X'$.
%
\end{lemma}


\subsection{The attaching-iteration method}

In this subsection we consider
pairs
of (not necessarily connected) spaces $X \subset Y$ and
maps
$f: Y \to Y$ which
satisfy
$f(X) \subset X$. Assume that $Y$ is obtained from the disjoint union of $X$ and a space  $Z$ by gluing a subspace $Z_0 \subset Z$ to $X$ via an attaching map $\phi : Z_0 \to X$:
$$Y = Z \cup X / < z = \phi(z) \mid z \in Z_0 >$$

Let $X'$ be an
$f$-invariant union of some of the connected components
of $X$, and let $Z'_0 \subset Z_0$ the subset consisting of those points that are glued via $\phi$ to $X'$.
We can now construct
a new space
in the following
way:  We fix an integer $t \geq 1$.  Then we unglue every
point $z' \in Z'_0$ from $X'$
and reglue it to
$f^t(z')$, to obtain a new space
$$Y_1  = Z \cup X / < z = \phi(z), z' = f^t \phi(z') \mid z \in Z_0 \smallsetminus Z'_0, z' \in Z'_0 > \, .$$
We define a
map
$h:
Y \to Y_1$ which restricts to the identity on the subspace $Z$ as well as on
$X \smallsetminus X'$, and maps every point
$x \in X'$ to the point $f^t(x)$.
It is easy to verify that these definitions are compatible with the gluing maps. We observe:

\begin{remark}
\label{homotopyequivalence}
If the restriction of
$f$ to a self-map of $X'$ is a homotopy equivalence,
then also the map $h$ is a homotopy equivalence.
\end{remark}

We now define a map $f_1: Y_1 \to Y_1$ as follows, where we distinguish three cases according to the position of the point $y_1 \in Y_1$ and to that of $f(y) \in Y$, where $y$ denotes the point ``corresponding'' to $y_1$ in the identical copy in $Y$ of the subspace $X \subset Y_1$ or $(Z \smallsetminus Z_0) \subset Y_1$. The case $y_1 \in Z_0$ can be discarded, as any such $y_1$ is identified via the map $\phi$ with some point of $X$.  We define:

\begin{enumerate}
\item
If $y_1 \in X' \subset Y_1$, then we set $f_1(y_1) = f(y)$.
\item
If $y_1 \in (Z \smallsetminus Z_0) \cup (X \smallsetminus X') \subset Y_1$ and $f(y) \in X' \subset Y$, then we set $f_1(y_1) = f^{t+1}(y)$.
\item
If $y_1 \in (Z \smallsetminus Z_0) \cup (X \smallsetminus X') \subset Y_1$ and $f(y) \in Y \smallsetminus X'$, then we set $f_1(y_1) = f(y)$.
\end{enumerate}

\medskip


These definitions extend continuously to
define a map on $Z_0$ which is
compatible with the gluing maps, and hence
one obtains
directly
a well defined
map $f_1: Y_1 \to Y_1$
with $f_1(X) \subset X$.
The proof of the following proposition is now an exercise:


\begin{proposition}
\label{commutingmaps}
Let $f: Y \to Y,  f_1: Y_1 \to Y_1$ and $h: Y \to Y_1$ be as above.

\begin{enumerate}

\item[(1)]
The maps $f$ and $f_1$ commute via $h$:
$$h f = f_1 h$$

\item[(2)]
If $f$ and $h$ are homotopy equivalences, then so is $f_1$.

\item[(3)]
If $Y$ is a graph-of-spaces
with vertex space collection $X$, and if
 $f$ is a partial train track map relative $X$, then so are $Y_1$ and $f_1$.
\end{enumerate}
\end{proposition}

\begin{remark}
\label{preessential}
Notice that the assumption in Remark \ref{homotopyequivalence}, that  $f$ restricts to a homotopy
equivalence of $X'$, is necessary in order to get a homotopy
equivalence $h$ as above, with $h f = f' h$. If one is content with a more
general map $h$ which satisfies
this equation, but is only an isomorphism on
$\pi_{1}$, then the weaker assumption suffices that $f$ induces a
$\pi_{1}$-isomorphism on each connected component of $X'$.
It is, however, unavoidable that  $X'$ contains no {\em inessential} component $X_v$
of $X$, i.e. $X_v$ satisfies $\pi_1 X_v = \{1\}$. Otherwise $\pi_{1} Y_{1}$ would be different
from $\pi_{1} Y$.
\end{remark}

In the context considered below it turns out that case 3. in the definition of the map $f_1$ before Proposition \ref{commutingmaps} does never occur.  In order to simplify the notation, we define:

\begin{definition}
\label{essential}
Let $f: \GG^2 \to \GG^2$ be a partial train track map with 2-cells relative $X \subset \GG^2$.  A connected component of $\GG^2$ is called {\em essential}, if it is mapped by $f$ via a homotopy equivalence to another connected component. A connected component is called {\em pre-essential} if it is mapped by $f$ to an essential component.
\end{definition}

\medskip

We now want to further specify the particular application of the attaching-iteration method that will be used
in the next subsection,
to construct strong partial train track maps via an iterative procedure.
To be specific, other than Proposition \ref{commutingmaps} we also use Proposition \ref{eventuallypseudolegal} and Remark \ref{pseudo-becomes-legal} (b) to obtain part (4) of the following:


\begin{corollary}
\label{attaching-iteration-for-tts}
%
Let $f': \GG'^2 \to \GG'^2$ be a
strong
partial train track map with 2-cells, relative to a subspace $X' \subset \GG'^2$, and assume that in $\GG'^2$ Nielsen faces have been expanded at every periodic INP. Assume furthermore that every connected component of $\GG'^2$ is either essential or pre-essential.

Let $\GG_0$ be a graph-of-spaces relative to $\GG'^2$,
where
$\GG_0$ is obtained from $\GG'^2$ by attaching a finite collection $\hat \Gamma_0$ of edges to $X'$.  Let $f_0: \GG_0 \to \GG_0$ be an extension of the map $f'$, which is a partial train track map relative to $\GG'^2$.

Then there exists a graph-of-spaces $\GG$, given by attaching a collection $\hat \Gamma$ of edges to a collection $X$ of vertex spaces, as well as a partial train track map $f: \GG \to \GG$ relative to $X$, which have the following properties:

\begin{enumerate}
\item[(1)]
There is a homotopy equivalence $h: \GG_0 \to \GG$ with $h(\GG'^2) \subset X$ such that $f_0 h$ is homotopic to $h f$. The map $h$ restricts to a homeomorphism $h_\Gamma: \hat \Gamma_0 \to \hat \Gamma$, with $f h(x) = h f_0 (x)$ for all points $x \in \hat \Gamma_0$ with $f_0(x) \in \hat \Gamma_0$.

\item[(2)]
In particular, if the partial train track map (rel. $\GG'^2$) $f_0$ satisfies the initial-segments condition, then so does $f$.

\item[(3)]
There is a homeomorphism $\psi: \GG'^2 \to X$ and an integer $t \geq 0$, such that
on every essential component of $\GG'^2$ the map $h$ is equal to $\psi f'^t $, while on every pre-essential component of $\GG'^2$ the map $h$ is equal to $\psi$. Moreover, the restriction $f_X: X \to X$ of $f$ is a strong partial train track map with 2-cells relative to $\psi(X')$, and $f_X$ is homotopic to $\psi f' \psi^{-1}$ on every essential component and to $\psi f'^{t+1} \psi^{-1}$ on every pre-essential component of $X$.

\item[(4)]
The image $f(e)$ of any edge $e$ of $\hat \Gamma$ is an alternate concatenation of subpaths that are either contained in $\hat \Gamma$, or else they are strongly legal paths in $X$.
\end{enumerate}
\end{corollary}

The above proposition is a crucial step in our iterative procedure given in subsection \ref{rel-tt-to-beta-tt} to built $\beta$-train track maps from relative train track maps. The reader should be warned, however, that 
the resulting map $f: \GG \to \GG$ is not necessarily yet a partial train track map relative to $\psi(X')$.
This conclusion would in general be wrong, despite of the fact that $f$ is
a partial train track map relative to $X$,
that the restriction of $f$ to $X$ is
a partial train track map relative to $\psi(X')$, and
that property (4) of the above proposition holds:  A priori, it may still happen that a path $f^t(e)$ is not reduced relative to $\psi(X')$, for some edge $e$ of $\hat \Gamma$ and some $t \geq 1$.

\subsection{Construction of a $\beta$-train track map from a relative train track map}
\label{rel-tt-to-beta-tt}

Bestvina-Handel have proved in \cite{BH} that for every automorphism $\alpha$ of $\FN$ there exists a finite connected graph $\Gamma$ with identification $\pi_1 \Gamma = \FN$, such that $\alpha$ can be represented by a relative train track map $g: \Gamma \to \Gamma$. This means,
using the terminology introduced
in the previous subsections,
that there is an $g$-invariant filtration $\Gamma_0 \subset \Gamma_1 \subset \ldots \Gamma_s = \Gamma$ of (not necessarily connected) subgraphs, where $\Gamma_0$ is the vertex set of $\Gamma$, such that
the following conditions are satisfied:

\begin{properties}
\label{relative-tt-properties}
For any $k \in \{1, \ldots, s\}$ we denote by $g_k$ the restriction of $g$ to the $g$-invariant subgraph $\Gamma_k \subset \Gamma$.
\begin{enumerate}
\item[(1)]
The map $g_k: \Gamma_k \to\Gamma_k$ is a partial train track map relative to $\Gamma_{k-1}$.
\item[(2)]
If the map $g_k$ is expanding (relative to $\Gamma_{k-1}$), then it satisfies the initial-segments condition (see Definition \ref{initial-segments}).
\item[(3)]
If the map $g_k$ is non-expanding (relative to $\Gamma_{k-1}$),
then either all of $\Gamma_k$ is mapped by $f_k$ to $\Gamma_{k-1}$, or else $g_k$ is, modulo $\Gamma_{k-1}$, a transitive permutation of the edges of $\Gamma_k \smallsetminus \Gamma_{k-1}$.
\end{enumerate}
\end{properties}

We now describe the construction that derives  from such a relative train track map $g: \Gamma \to \Gamma$ a $\beta$-train track map $f: \GG^2 \to \GG^2$.  In a first attempt we concentrate on the weaker property that $f$ is a strong partial train track map with 2-cells. Subsequently we show that the construction defined below yields indeed a map that satisfies the additional properties claimed in Theorem \ref{betterttrepresentative}.

\smallskip

Our construction
%
%
proceeds iteratively, moving at each step one level up, i.e. from $g_{k-1}$ to $g_k$.
To start this iterative process, note that for $k = 1$ the map $g_1: \Gamma_1 \to \Gamma_1$ is an absolute train track map (in the sense of \cite{BH}), since the relative part $\Gamma_0$ consists precisely of the vertices of $\Gamma$.  Thus $g_1$ is in particular a  strong partial train track map with 2-cells (where the retraction $r$ is simply the identity map).

\smallskip

Let us now assume, by induction, that there is a strong partial train track map with 2-cells $f_{k-1}: \GG^2_{k-1} \to \GG^2_{k-1}$, relative to a subspace $X_{k-1}$,
as well as a homotopy equivalence $h_{k-1}: \Gamma_{k-1} \to \GG^2_{k-1}$ such that $h_{k-1} g_{k-1}$ is homotopic to $f_{k-1} h_{k-1}$. We then consider a copy $\hat \Gamma$ of the edges of $\Gamma_k \smallsetminus \Gamma_{k-1}$, and we attach each edge $\hat e$ of $\hat \Gamma$, with copy $e$ in $\Gamma_k \smallsetminus \Gamma_{k-1}$, at the points $h_{k-1}(\partial e)$ to $\GG^2_{k-1}$. We define the map $h_k$ on $\Gamma_k$ to agree with $h_{k-1}$ on $\Gamma_{k-1}$, and every edge $e$ of $\Gamma_k \smallsetminus \Gamma_{k-1}$ is mapped by $h_k$ to its copy $\hat e$ in $\hat \Gamma$.

We define the map $f_k$ on $\hat \Gamma \cup \GG^2_{k-1}$ to agree with $f_{k-1}$ on $\GG^2_{k-1}$, and for each edge $\hat e \in \hat \Gamma$ we define $f_k(\hat e)$ to be
the concatenation $\gamma_0 \circ h_k g_k(e) \circ \gamma_1$. Here the $\gamma_i$ are the paths traced out by the points $h_{k-1}(\partial e)$ during the homotopy between $h_{k-1} g_{k-1}$ and $f_{k-1} h_{k-1}$.
 Thus we obtain a graph-of-spaces $\GG_k = \hat \Gamma \cup \GG^2_{k-1}$ with vertex space collection $\GG^2_{k-1}$, a map $f_k: \GG_k \to \GG_k$ with $f_k(\GG^2_{k-1}) \subset \GG^2_{k-1}$, and a homotopy equivalence $h_k: \Gamma_k \to \GG_k$ with $h_k(\Gamma_{k-1}) \subset \GG^2_{k-1}$, such that $h_k g_k$ is homotopic to $f_k h_k$.

\smallskip

As first step in our iterative construction
we expand Nielsen faces in $\GG_{k-1}^2$, until at all periodic INP's of the partial train track map $f_{k-1}$ rel. $X_{k-1}$ there is a Nielsen face attached. All auxiliary edges introduced in this procedure are added to the relative subspace $X_{k-1}$.
We extend the strong deformation retraction $r_{k-1}$ on $\GG^2_{k-1}$, which exists by induction, by precomposing it with the
strong deformation retraction $\hat r_k$ which is defined as is the map $\hat r$
in Definition \ref{Nielsen-faces} (c):
every auxiliary edge is pushed over the corresponding Nielsen face, and 
any of the original points of $\GG^2_{k}$
is left fixed.

\smallskip
Next we apply Corollary \ref{attaching-iteration-for-tts}, to obtain that
the image $f_k(\hat e)$ of any edge $\hat e$ in $\hat \Gamma$ is an alternating concatenation of subpaths in $\hat \Gamma$ and of strongly legal paths in $\GG^2_{k-1}$ (relative to $X_{k-1}$).
For simplicity we keep the same names, thus suppressing notationally the homotopy equivalence $h$ as well as the homeomorphisms $h_\Gamma$ and $\psi$ from Corollary \ref{attaching-iteration-for-tts}.


\begin{proposition}
\label{strong-tt-representative}
The resulting map $f_k$ is (after a homotopically irrelevant modification in Case 4 of the proof) a strong partial train track map $f_k: \GG^2_k \to \GG^2_k$ with 2-cells, relative to the subspace $X_k \subset \GG^2_k$.

The relative part $X_k$ is equal to $X_{k-1}$, except for Case 3 of the proof, where $X_k$ is specified
in the proof.
\end{proposition}

\begin{proof}
We will distinguish four cases as follows:

\smallskip
\noindent
{CASE 1:}  Assume that $f_k$ is expanding (relative to $\GG^2_{k-1}$). In this case we have the initial-segments condition given as hypothesis
by Property \ref{relative-tt-properties} (2), so that we can apply
Lemma \ref{exponential-stratum}
to obtain directly the statement of the above proposition, for
$X_k := X_{k-1}$.

\smallskip
\noindent
{CASE 2:}  If $f_k$ is not expanding (relative to $\GG^2_{k-1}$), and if $f_k(\hat \Gamma)$ is entirely contained in $\GG^2_{k-1}$ but not in $X_{k-1}$, then
after the above application of Corollary \ref{attaching-iteration-for-tts} each of the paths $f_k(\hat e)$ is a strongly legal path in $\GG^2_{k-1}$, so that
again the claim follows directly, for
$X_k := X_{k-1}$.

\smallskip
\noindent
{CASE 3:}  If $f_k$ is not expanding (relative to $\GG^2_{k-1}$), and if $f_k(\hat \Gamma)$ is entirely contained in $\hat \Gamma \cup X_{k-1}$, we define $X_k = \hat \Gamma \cup X_{k-1}$
and obtain again directly the above claim.
Note that in this case
the train track part of $\Gamma_k$ grows polynomially under iteration of $f$.

\smallskip
\noindent
{CASE 4:}  Assume that $f_k$ is not expanding (relative to $\GG^2_{k-1}$), but that $f_k(\hat \Gamma)$ is not entirely contained in $\GG^2_{k-1}$, and also not in $\hat \Gamma \cup X_{k-1}$.
In this case
for every edge $e_i$ of $\hat \Gamma$ one has $f_k(e_i) =\gamma^i_0 \circ e^{\pi(i)} \circ \gamma_1^i$, for some permutation $\pi$ of the edges of $\hat \Gamma$ and legal paths $\gamma^i_0, \gamma^i_1$ in $\GG^2_{k-1}$ relative to $X_{k-1}$.
%
Furthermore, not all of the $\gamma^i_0, \gamma^i_1$ are contained in $X_{k-1}$ (or else we would be in Case 3).
%
Thus we can apply Lemma \ref{one-edge} to define a modification of $\GG_k$ and $f_k$ (which is homotopically trivial rel. $X_{k-1}$), and with this modification our claim is now proved by Lemma \ref{one-edge},
for $X_k := X_{k-1}$.
\end{proof}



\medskip

We now verify inductively that the additional properties from Theorem \ref{betterttrepresentative} are satisfied, i.e. that the map $f_k$ is indeed a $\beta$-train track map. To be precise, for this purpose one has first to expand further Nielsen faces in $\GG^2_k$
until
at every periodic INP of $f_k$ a Nielsen face is expanded. (Note that this is anyway the first modification in the above described iterative construction to produce strong partial train tracks with 2-cells, when passing to the next level, with index $k+1$).

\smallskip

In particular, this expansion of Nielsen faces, at every periodic INP, has to be done at the very last level of our iterative procedure, to obtain from $f_s: \GG^2_s \to \GG^2_s$ a {\em $\beta$-train track representative} $f: \GG^2 \to \GG^2$ of the same automorphism of $\FN$ that was originally represented by the relative train track map $g: \Gamma \to \Gamma$.

\smallskip

We state the following theorem for the map $f$, but we prove it via induction by passing from $f_{k-1}$ to $f_k$.




%

\begin{theorem}
\label{additional-properties}
The strong partial train track map with 2-cells $f: \GG^2 \to \GG^2$, relative to $X$, is a { $\beta$-train track map}, in that it has the following additional properties:
\begin{enumerate}
\item[(1)]
The map $f$ is expanding relative to $X$.

\item[(2)]
The map $f$ has polynomial growth on the relative part $X$.

\item[(3)]
Every 2-cell of $\GG^2$ is a Nielsen face that has been expanded at a periodic INP of some of the strong partial train track maps $f_k: \GG^2_k \to \GG^2_k$ rel. $X_{k}$, in the iterative construction for any of the steps $k = 1, \ldots,  s$.

\item[(4)]
The
above defined retraction $r = r_{s} \circ  \hat r$ maps $\GG^2$ to a
subgraph $\Gamma'$ of $\GG^2_k$,
such that the restriction $r \circ f|_{\Gamma'}: \Gamma' \to \Gamma'$ is a relative train track map with respect to the iteratively defined filtration $\Gamma'_0 \subset \Gamma'_1 \subset \ldots \subset \Gamma'_s$.
The graph $\Gamma'$ is obtained from the 1-skeleton of $\GG^2$ by omitting all auxiliary edges introduced when expanding a Nielsen face for any of the intermediate
partial train track maps $f_k: \GG_k \to \GG_k$ rel. $\GG^2_{k-1}$, for $k = 1, \ldots,  s$.

\item[(5)]
Every path $\gamma$ in $\GG^2$
has an iterate $f^t(\gamma)$ which is homotopic rel. endpoints to a strongly legal path.
There is an upper bound to the exponent $t$, which depends only on the number $q$ of factors in any decomposition of $\gamma = \gamma_1 \circ \ldots \circ \gamma_q$ as concatenation of strongly legal paths $\gamma_i$, and not on the particular choice of $\gamma$ itself.

The analogous statement is true for free homotopy classes of loops in $\GG^2$.

\item[(6)]
The relative part $X$ is a graph, and the marking map (given for example by the retraction $r$ together with the identification $\pi_1 \Gamma' = \FN$) restricts on each connected component to a monomorphism.

\item[(7)]
%
Strongly legal paths in $\GG^2$, or strongly legal paths with $r$ applied to some of its subpaths, lift to quasi-geodesics in the universal covering $\tilde \GG^2$.
\end{enumerate}
\end{theorem}

\begin{proof}
(1)
This property is actually part of the fact that $f_k$ is strong, see Definition \ref{strong-tts}.
However, to be explicit, we note
that
in Cases 1 and 3 this property follows directly from the induction hypotheses, as
%
the train track part of $\GG^2_k$ is equal to that of $\GG^2_{k-1}$. In Cases 2 and 4,
the expansiveness of $f_k$ is a direct consequences of the
particular properties stated
in each of those cases. Please note that in Case 4 we
need to use the hypothesis from
Property \ref{relative-tt-properties} (3)
that the permutation $\pi$ is transitive.

\smallskip
\noindent
(2) To show this, we first observe that all auxiliary edges added to $X_{k-1}$, in the first step of our iterative construction,
are permuted by $f_k$ among themselves.
Thus, in Cases 1, 2 and 4 we can use directly the inductive hypothesis that $f_{k-1}$ has polynomial growth on
its relative part,
since in these three cases,
up to adding the auxiliary edges, the relative part has not been changed when passing from $f_{k-1}$ to $f_k$.
In Case 3 we observe
%
that the edges added to $X_{k-1}$ are also permuted among themselves, up to adding initial and terminal subpaths to them which are entirely contained in $X_{k-1}$, so that again the inductive hypothesis about polynomial growth of $f_{k-1}$ suffices to derive the claim.


\smallskip
\noindent
(3)
This follows directly from the definition of our iterative construction of $f_k$ and $\GG^2_k$.

\smallskip
\noindent
(4)
To see this, we first observe that before stating Proposition \ref{strong-tt-representative} and considering the 4 cases, we applied the attaching-iteration method through Corollary \ref{attaching-iteration-for-tts} to $f_k: \GG_k \to \GG_k$.
It is easy to see that
any time one applies
the attaching-iteration method to $f_k$, one can simultaneously apply the same operations to the given relative train track map $f: \Gamma \to \Gamma$ (or, more precisely, to $g_k: \Gamma_k \to \Gamma_k$), and the resulting map is again a relative train track map.
Thus we can use by induction that the claim is true for $f_{k-1}$ after having applied Corollary \ref{attaching-iteration-for-tts},
and thus obtain the claim for $f_k$ directly from the definitions in each of the 4 above cases.
%

\smallskip
\noindent
(5)
By Remark \ref{pseudo-becomes-legal} (a) and Proposition \ref{eventuallypseudolegal} (a) there is an iterate $f_k^t(\gamma)$ that is homotopic to a pseudo-legal path. Since in $\GG^2_k$ Nielsen faces have been expanded at every periodic INP, the pseudo-legal path can be homotoped,
at each of its periodic INP's, over the corresponding Nielsen face, to give after finitely many of such alterations a strongly legal path that is homotopic to $f_k^t(\gamma)$.

It follows directly from Proposition \ref{eventuallypseudolegal} (b) that the number $t \geq 0$ of iterates of $f_k$, needed above to make $f_k^t(\gamma)$ pseudo-legal, can be bounded above as function of the number of illegal turns in the originally given path $\gamma$.

\smallskip
\noindent
(6)
The fact that $X_k$ is a graph follows directly from the induction hypothesis that $X_{k-1}$ is a graph, since in the above inductive procedure only edges have been added. The (not really essential) fact that the marking map is injective on each component, however, would in general be wrong, unless we actually introduce, instead of auxiliary edges, an auxiliary vertex and auxiliary half edges, as explained in the Aside \ref{doubleauxiliary}. For more detail see Definition 3.7 of \cite{Lu1}.


\smallskip
\noindent
(7)
The retraction $r_k: \GG^2_k \to \Gamma'_k \subset \GG^2_k$ is a deformation retraction and as such
homotopic (in $\GG^2_k$) to the identity map of $\GG^2_k$, and it maps strongly legal paths to reduced paths in $\Gamma'_k \subset \GG^2_k$. Since reduced paths in $\Gamma'_k$ lift to geodesics in the universal covering $\tilde \Gamma'_k$ and thus to quasi-geodesics in $\tilde \GG^2_k$, it follows that strongly legal paths, or strongly legal paths with $r_k$ applied to some of its subpaths,
also
lift
to quasi-geodesics in $\tilde \GG^2_k$.
\end{proof}

\subsection{The structure of automorphisms of $\FN$}
\label{structure-of-autos}



Partial train track maps with Nielsen faces as introduced in \cite{Lu1}, and hence in
particular the $\beta$-train track maps considered here, have
a
crucial advantage over
all other train tracks, classical \cite{BH} or improved \cite{bfhtits1}
or improved-improved \cite{BG3},
\cite{Feighn-Handel},
etc: The structure of the train track transition matrix $M(f) = (m_{e,
e'})_{e, e' \in \hat \Gamma}$ is an
invariant of the conjugacy class of the outer automorphism $\hat
\alpha \in \Out \FN$
defined by $\alpha$. Here the coefficient $m_{e, e'}$ is given by the
number of times that the (legal) path $f(e')$ crosses over $e$ or
its inverse $\bar e$. The following result has been shown in
\cite{Lu1}, \S 4.  For a reader friendly exposition of train tracks,
invariant $\R$-trees, and the precise relationship to the
transition matrix and its eigen vectors, see \cite{lu2}.

\begin{theorem}
\label{structuretheorem}
(a)
For any $\beta$-train track representative $f:\GG^{2} \to \GG^{2}$
of $\alpha \in \Aut \FN$
there is a canonical bijection between the set of $\alpha$-invariant $\R$-trees
$T$ as given in Proposition \ref{invarianttree} (a) and the set of
row eigen
vectors $\vec v_{*}$
of $M(f)$ with real eigen value $\lambda > 1$.

\smallskip
\noindent
(b)  If $T$ is given by the eigenvector $\vec v_{*}$ as above, then
every conjugacy class of non-trivial point stabilizers in $T$,
unless it is of polynomial $\alpha$-growth,
is given by
a non-trivial $M(f)$-invariant subspace of $\R^{\hat \Gamma}$ on which
$\vec v_{*}$ has coefficients of value 0.
These invariant subspaces are in 1-1 relationship with
those
complementary
components of the support of $\vec v_{*}$ in $\GG^{1}$
that are not contained in $X$.
In particular, the induced
automorphism on these point stabilizers is represented by a
sub-train-track of $\GG^{2}$, given by those complementary
components, provided with the
corresponding
restriction of the
train track map $f$.
\end{theorem}

The use of this structure theorem is highlightened by the fact that,
after replacing $f$ by a suitable power, there are (up to scalar
multiples) finitely many eigen vectors of $M(f)$ which have as
support a subspace of $\R^{\hat \Gamma}$ on which $M(f)$ has an
irreducible matrix with irreducible powers. Here ``irreducible''
refers to the standard use of this terminology
in the context of non-negative matrices.
The resulting invariant $\R$-trees, called {\em partial pseudo-Anosov}
trees in \cite{Lu1}, are the smallest building blocks out of which the
exponentially growing part of $\alpha$ is iteratively built.
[However, a word of caution seems to be appropriate here:
Even if $M(f)$ consists of a
single irreducible block with irreducible powers, and if the relative
part of $\GG^{1}$ is empty, one can not conclude that $\alpha$ is an
iwip automorphism. This conclusion is only possible after a further
local analysis at the vertices of $\GG^{1}$, see \cite{lu2}, \S 7 and
\cite{JL}, \S IV.]

\smallskip


In \S 3  iteratively constructed invariant trees  $T_j$ have been considered, in order to find the characteristic family ${\cal H}(\alpha)$ of (conjugacy classes of) subgroups $H_i$ where $\alpha$ has polynomial growth. All of these $T_j$ are given as in Theorem \ref{structuretheorem} by eigenvectors $\vec v_*$ of the transition matrix $M(f)$ or of ``submatrices'' of $M(f)$ describing the induced $\beta$-train track map on an $f^t$-invariant subgraph of $\GG^2$.
In particular, we obtain:

\begin{proposition}
\label{relative-components}
The connected components $X_i$ of the relative part $X$ of $\GG^2$ are in 1-1 correspondence with the subgroups $H_i$ from the canonical family ${\cal H}(\alpha)$ of polynomial $\alpha$-growth:  one has
$$H_i = \pi_1 X_i\, ,$$
up to conjugation and permutation of the $H_i$.
\end{proposition}

\end{document}